\documentclass[a4paper]{article}

\usepackage[english]{babel}
\usepackage[utf8x]{inputenc}
\usepackage[T1]{fontenc}

\usepackage[a4paper,top=3cm,bottom=2cm,left=3cm,right=3cm,marginparwidth=1.75cm]{geometry}

\usepackage{mathptmx,amsmath,colonequals}
\usepackage{amssymb,gensymb,mathrsfs}
\usepackage{amscd,amsthm}
\usepackage{mathtools}
\usepackage{graphicx}
\usepackage[colorinlistoftodos]{todonotes}
\usepackage[colorlinks=true, allcolors=blue]{hyperref}
\usepackage[shortlabels]{enumitem}
\usepackage{extarrows}
\usepackage{tikz,tikz-cd}
\usepackage[ruled,linesnumbered]{algorithm2e}
\usepackage{listings}
\usepackage{nicefrac}
\usepackage{array}
\usepackage{booktabs}
\usepackage{float}

\SetKwInput{KwInput}{Input}
\SetKwInput{KwNotation}{Notation}
\SetKwInput{KwOutput}{Output}
\DontPrintSemicolon
\newcommand{\code}[1]{\textnormal{\texttt{#1}}}

\newtheorem{thm}{Theorem}[section]
\newtheorem{lem}[thm]{Lemma}
\newtheorem{cor}[thm]{Corollary}
\newtheorem{prop}[thm]{Proposition}

\newtheorem{warn}[thm]{Warning}

\newtheorem{heur}[thm]{Heuristic}
\newtheorem{nota}[thm]{Notation}
\newtheorem{quest}[thm]{Question}
\theoremstyle{definition}
\newtheorem{defn}[thm]{Definition}
\newtheorem{ex}[thm]{Example}

\theoremstyle{remark}
\newtheorem{rmk}[thm]{Remark}

\newcommand{\isom}{\cong}
\newcommand{\ins}{\subset}

\newcommand{\Hom}{\operatorname{Hom}}

\newcommand{\Spec}{\operatorname{Spec}}

\newcommand{\Proj}{\operatorname{Proj}}

\newcommand{\triv}{\texttt{TRIV}}
\newcommand{\merge}{\texttt{MERGE}}
\newcommand{\wicsalg}{\texttt{WICS}}

\newcommand{\exps}{\operatorname{exps}}
\newcommand{\wics}{\operatorname{wics}}
\newcommand{\coeff}{\operatorname{coeff}}
\newcommand{\length}{\operatorname{length}}

\newcommand{\modp}{\; (\text{mod} \; p)}
\newcommand{\mathmod}{ \text{ mod }}

\title{K3 surfaces of any Artin--Mazur height over \(\mathbb{F}_{5}\) and 
    \(\mathbb{F}_{7}\) via 
quasi-\(F\)-split singularities and GPU acceleration}
\author{Ryan Batubara, Jack J. Garzella, Alex Pan}

\begin{document}

\maketitle

\begin{abstract}
    \noindent We develop a fast algorithm to calculate the Artin-Mazur height (equivalently, the
    quasi-\(F\)-split height) of a Calabi-Yau hypersurface, building on the work in
    \cite{kty-2022-fedder}. 
    We provide an implementation of our approach, and use it to show that there are quartic
    K3 surfaces of any Artin-Mazur height over \(\mathbb{F}_{5}\) and \(\mathbb{F}_7\).
\end{abstract}

\section{Introduction}

Given an algebraic variety \(X\) over a field of characteristic \(p\),
Artin and Mazur in \cite{artin-mazur-1977-height}
define the height \(h(X) \in \mathbb{N} \cup \{\infty\}\) 
(sometimes written \(h\) if \(X\) is obvious from the context)
to be the height 
of a certain formal group associated to \(X\).
The height is in some sense a measure of the
``arithmetic complexity'' of a variety,
with larger heights indicating ``more complexity''.
For example, consider the case when \(X\) is a 
K3 surface (i.e. a surface for which \(\omega_{X} \isom \mathcal{O}_{X} \) 
and \(H^{1}(X,\mathcal{O}_{X} ) = 0\)).
In this case, \(h\) determines the \textit{Newton polygon}
of \(X\), which in turn gives partial information about the 
point counts of \(X\) over finite fields,
i.e. \(\#X(\mathbb{F}_{q})\) for \(q = p^{r}, r \in \mathbb{N}\).
One may construct another polygon, the \textit{Hodge polygon},
from the Hodge-theoretic data associated to \(X\).
For K3 surfaces (and in fact for any surfaces), 
the Newton polygon equals the Hodge polygon
(i.e. the variety is ordinary) if and only if \(h=1\),
and the Newton polygon is \textit{supersingular} if and only if 
\(h = \infty\).

Given a field \(k\) of characteristic \(p\), one may ask

\begin{quest}
    For which values \(h\) does there exist a variety
    \(X\) such that \(h(X) = h\)?
\end{quest}

In the case that \(X\) is a K3 surface, it is known from the 
original work of Artin and Mazur
that either \(1 \leq h(X) \leq 10\) or \(h = \infty\).
One expects that all such $h$ are realized by K3 surfaces
in characteristic $p$.
For example, in \cite{artin-1974-k3-surfaces},
Artin shows\footnote{
    Artin showed this conditionally, dependent on flat
    duality for surfaces, which was later proved
    by Milne in \cite{milne-1976-flat-duality}.
}
that this holds over an algebraically
closed field \(k\) of characteristic \(p\) in the
case that \(p \equiv 3 \mathmod 4\).
In \cite{taelman-2016-k3-given-l-function}, 
Taelman shows that
all possible \(h\) are realized
over some sufficiently big finite field \(\mathbb{F}_{q}\). 
By base change, this implies the result for extensions
thereof (e.g. algebraically closed fields).
However, Taelman's argument is not constructive,
and there is no known bound on how big \(q\) must be to guarantee
existence of surfaces of any height.

More concretely, it follows from the work of
Kedlaya and Sutherland in 
\cite{kedlaya-sutherland-2016-census-k3-f2}
that there exist quartic K3 surfaces
(i.e. quartic surfaces in \(\mathbb{P}^{3}\))
of all possible heights over \(\mathbb{F}_{2}\).
More recently,
Kawakami, Takamatsu, and Yoshikawa in \cite{kty-2022-fedder}
have given examples of quartic K3 surfaces
of all possible heights over \(\mathbb{F}_{3}\).

Kawakami, Takamatsu, and Yoshikawa consider, 
instead of the Artin--Mazur height, another quantity called the
\textit{quasi-\(F\)-split height} 
(see Section 2 for a definition),
which is known to be equal to the Artin--Mazur height
for Calabi-Yau varieties, and thus K3 surfaces.
The theory of the quasi-\(F\)-split height has its
roots in the theory of \(F\)-singularites and commutative
algebra/birational geometry in characteristic \(p\).
In particular, one theorem fundamental
to the development of the theory of \(F\)-singularities
is Fedder's criterion (Theorem \ref{thm:fedder:criterion}), 
which gives a very concrete and computable way to check
whether or not a variety has height \(1\).
The main theorem of Kawakami, Takamatsu, and Yoshikawa
(\cite[Theorem~A]{kty-2022-fedder})
is a generalization of Fedder's criterion
to higher heights which gives a 
computable way to check whether a variety has height \(h\).
Moreover, they provide a simpler version
(\cite[Theorem~C]{kty-2022-fedder})
in the case when \(X\) is a Calabi-Yau hypersurface
(and thus also for quartic K3 surfaces).
The forthcoming work 
\cite{fgmqt-2025-witt-vectors-macaulay2} 
provides an implementation of Theorem A.
However, the algorithm used to compute the examples in
\cite{kty-2022-fedder}
is not especially practical
as it requires multiplying many large polynomials
(\cite{takamatsu-2024-algorithm}, 
\cite{fgmqt-2025-witt-vectors-macaulay2}),
limiting the computation to
very low characteristic.

In this work, we push the method of Kawakami,
Takamatsu, and Yoshikawa to the limits of readily
available hardware, 
and produce examples of quartic K3 surfaces \(X\) with all
possible heights \(h\) 
over \(\mathbb{F}_{5}\) and \(\mathbb{F}_{7}\).
Our main insight is that one of the main
bottlenecks of the problem can be broken down 
into many repeated matrix-vector
multiplications for a square matrix of size 
\(\binom{4p-1}{3}\).
Because of this, we can use Nvidia's CUBLAS library 
\cite{nvidia-2024-cublas}
to perform the matrix multiplications
on a GPU, making such computations
nearly instantaneous.
To create this matrix, one must calculate the matrix of a
certain linear operator.
We provide a few novel algorithms which accomplish this task,
including some that can be implemented on the GPU.
The remaining algorithm is bottlenecked by the operation
of raising a polynomial to a large power, 
so we implement this on the GPU as well using 
a Fast Fourier Transform (FFT) approach.
Our algorithms and implementations lead to a high throughput
of heights of K3 surfaces: 
about 1400 surfaces per second over \(\mathbb{F}_{5}\) and
about 180 surfaces per second over \(\mathbb{F}_{7}\).
Our implementations can also handle \(\mathbb{F}_{11}\) 
and \(\mathbb{F}_{13}\). 
For \(p = 11\), our method is comparable to the state of the art
for computing Newton polygons (e.g. 
\cite{chk-2019-toric-controlled-reduction}),
though it is noticably slower 
for \(p = 13\) (see section \ref{sec:heights:surfaces}).

All of our algorithms are implemented in Julia 
\cite{julia-2017}, and
make use of the OSCAR computer algebra system 
\cite{OSCAR-book}.
Code from this work is open source and available online 
in various
Julia packages: 
MMPSingularities.jl \cite{mmpsingularities-jl},
GPUFiniteFieldMatrices.jl \cite{gpuffmatrices-jl}, 
GPUPolynomials.jl \cite{gpupolynomials-jl},
and CudaNTTs.jl \cite{cudantts-jl}.

In Section 2, we give background on the quasi-\(F\)-split height
and Fedder's criterion.
In Section 3, we describe the naive implementation of
\cite[Theorem~C]{kty-2022-fedder} and describe
our modification.
In Section 4, we describe our algorithms for calculating the matrix
of the key linear operator which we term ``multiply then split''.
In Section 5, we describe our GPU implementation of the Number Theoretic
Transform, which is a finite field variant of the FFT.
In Section 6, we describe the considerations we need to keep in mind
to use CUBLAS.
In Section 7, we describe the computation of surfaces of all possible
heights over \(\mathbb{F}_{5}\) and \(\mathbb{F}_{7}\).

Throughout the paper, we do various timing experiments to compare 
approaches for each computational step.
All timing tests were performed with 
an Intel i5-8400 CPU and a Nvidia GeForce RTX 3070 GPU.

\subsection{Acknowledgements}

The authors thank the maintainers and support staff 
of the UCSD research cluster for use of their devices,
and Wilson Cheung for much helpful tech support.
Furthermore, they thank the authors of 
\cite{kty-2022-fedder} 
and \cite{fgmqt-2025-witt-vectors-macaulay2}
for providing preliminary implementations of 
Theorem C and Theorem A
of \cite{kty-2022-fedder}.
The second author thanks Jakub Witaszek, Kiran S. Kedlaya,
Steve Huang, and Karl Schwede for helpful discussions.
The authors also very grateful to 
Tommy Hofmann, Steve Huang, and 
Michael Monagan for giving feedback on an early draft of
this paper.

The second author was partially supported by the 
National Science Foundation Graduate Research
Fellowship Program under Grant No. 2038238, and a fellowship
from the Sloan Foundation.

\section{Preliminaries: the quasi-\(F\)-split height and Fedder's criterion}

\subsection{Witt vectors}

We begin by defining the ring of Witt vectors. 
Since the theory of Witt vectors is a vast and active topic of research,
we recall the bare minimum 
required for the quasi-\(F\)-split Fedder's criterion algorithm.
For a more complete introduction to the ring of Witt vectors with proofs, 
see \cite{rabinoff-2014-witt-vec}.
For an intuitive introduction or derivation of the ring of Witt vectors, see
\cite{kim-2017-witt-vec}.
There are many other perspectives on Witt vectors. 
For example:
\cite[Chapter~17]{hazewinkel-1978-formal-groups} covers the ring of Witt vectors
and its relationship with formal groups; 
\cite{kedlaya-2021-prismatic} gives a categorical perspective on the ring of Witt vectors
that relates to lifts of the Frobenius; and
\cite[Chapter~1]{schneider-2017-galois-rep-phi-gamma} defines a generalization
known as the ring of \textit{ramified Witt vectors} in detail.

\begin{defn}
	The \textit{\(n\)-th Witt Polynomial} \(\omega_{n}\) is defined as
	\[
		\omega_{n}(X_{0}, \ldots, X_{n}) = X_{0}^{p^{n}} + pX_{1}^{p^{n-1}} + \ldots + p^{n}X_{n}
        \in \mathbb{Z}[X_{0}, \ldots, X_{n}]
	\] 
    
\end{defn}

\noindent Now let \(R\) be a ring of characteristic \(p\).
We define the map \(\Phi\) to be the map 
\begin{align*}
    \Phi : \prod_{n \in \mathbb{N}}^{R} R  
    &\xlongrightarrow{(\omega_{n})_{n}} 
    \prod_{n \in \mathbb{N}}^{R} R \\
    (r_{0}, r_{1}, \ldots, r_{n}, \ldots) 
    &\xmapsto{\hphantom{(\omega_{n})_{n}}}
    (\omega_{0}(r_{0}), \omega_{1}(r_{0}, r_{1}), 
    \ldots, \omega_{n}(r_{0}, \ldots, r_{n}), \ldots).
\end{align*}

defined as \(\omega_{n}\) for the \(n\)-th component.
That is, \(\omega_{n}\) is considered as a 
map (of sets) \(R^{n} \xrightarrow{} R\) by
evaluation.

\begin{lem}
	There exist integer polynomials 
	\(S_n(X_{0}, \ldots, X_{n}, Y_{0}, \ldots, Y_{n})\) 
	with the property that 
	\[
		\Phi((S_{n})_{n \in \mathbb{N}}) =
		\Phi((X_{n})_{n \in \mathbb{N}}) 
		+ \Phi((Y_{n})_{n \in \mathbb{N}})
	.\] 
	Likewise, there exist integer polynomials 
	\(P_{n}(X_{0}, \ldots, X_{n}, Y_{0}, \ldots, Y_{n})\)
	such that
	\[
		\Phi((P_{n})_{n \in \mathbb{N}}) =
		\Phi((X_{n})_{n \in \mathbb{N}}) 
		\cdot \Phi((Y_{n})_{n \in \mathbb{N}})
	.\] 
\end{lem}

\begin{proof}
	For example, see \cite[Theorem~2.6]{rabinoff-2014-witt-vec} and the 
	surrounding discussion.
\end{proof}

We now define the ring of Witt vectors 
\(W(R)\) to be 
\(\prod_{n \in \mathbb{N}}^{} R \) 
as a set, with the ring structure defined
by 
\[
	(a_{n})_{n \in \mathbb{N}} + 
	(b_{n})_{n \in \mathbb{N}} =
	(S_{n}(a_{0}, \ldots, a_{n}, b_{0}, \ldots, b_{n}))_{n \in \mathbb{N}}
\] 
and likewise, multiplication is defined
using the \(P_{n}\).
The lemma then shows that \(\Phi\) is a homomorphism 
\[
	W(R) \xrightarrow{} \prod_{n \in \mathbb{N}}^{} R 
.\] 
The fact that the polynomials \(S_{n}\) and \(P_{n}\) do not depend on the base
ring \(R\) means that the construction is functorial; that is,
for a map of rings \(R \xrightarrow{} R^{\prime} \), we get a 
map \(W(R) \xrightarrow{} W(R^{\prime})\).

The main example, which also provides the fundamental motivation
for Witt vectors, is the case when \(R = \mathbb{F}_{p}\).
It is well known that \(W(\mathbb{F}_{p}) = \mathbb{Z}_{p}\), 
giving an alternative construction of the \(p\)-adic numbers.

\begin{warn}
	If one takes a
    naive \(p\)-adic expansion 
    \(\sum_{n = 0}^{\infty} c_{n}p^{n} \in \mathbb{Z}_{p}\)
    with \(c_{n} \in \{0, \ldots, p-1\}\), this does not correspond
    to the Witt vector \((c_{0}, c_{1}, c_{2}, \ldots)\).
    In fact, the aforementioned sum corresponds to 
    \((c_{0}, c_{1}^{p}, c_{2}^{p^{2}}, \ldots)\). See
	\cite[Section~2]{kim-2017-witt-vec}. 
    This motivates the following: 
\end{warn}

\begin{defn}
	There exists a homomorphism \(F \colon W(R) \xrightarrow{} W(R)\),
	called the \textit{Frobenius},
	defined by
	\[
	 (c_{0}, c_{1}, \ldots) \mapsto (c_{0}^{p}, c_{1}^{p}, \ldots)
	\] 
    (this is induced by the Frobenius on \(R\) by functoriality).
\end{defn}

\begin{defn}
	There exists a homomorphism \(V \colon W(R) \xrightarrow{} W(R)\),
	called the \textit{Verschiebung},
	defined by
	\[
		(c_{0}, c_{1}, c_{2}, \ldots) \mapsto  (0, c_{0}, c_{1}, \ldots)
	.\] 
\end{defn}

We see that \(W(R) / VW(R) \isom R\).
This is sometimes called the first \textit{truncated}
Witt vectors.
We also have higher truncated variants.

\begin{defn}
	The \(n\)-th truncated Witt vectors
	\(W_{n}(R)\) are defined as
	\(W(R) / V^{n}W(R)\).
\end{defn}

We may regard the \(n\)-th truncated Witt Vectors as consisting
of elements \((c_{0}, \ldots, c_{n-1})\) with addition
and multiplication by the \(S_{i}\) and \(P_{i}\).
Thus, one has \(\lim_{n} W_{n}(R) = W(R)\).

\begin{lem}
    \label{lem:frob:versch}
	The composition \(F \circ V = V \circ F\) is the multiplication by
	\(p\) map on \(W(R)\).
\end{lem}

\begin{proof}
	\cite[Proposition~5]{kim-2017-witt-vec}
\end{proof}

\begin{lem}
    \label{lem:perfect:witt}
    Let \(R\) be a \textit{perfect} ring,
    that is, a ring for which \(F\) is an isomorphism.
    Then \(W_{n}(R) = W(R) / p^{n}W(R)\).
\end{lem}

\begin{proof}
    Since \(F\) is an isomorphism, this follows from 
    Lemma \ref{lem:frob:versch}.
\end{proof}

\begin{rmk}
    If we take \(R = \mathbb{F}_{p}\), then
    Lemma \ref{lem:perfect:witt}
    shows that \(W_{n}(\mathbb{F}_{p}) = \mathbb{Z} / p^{n}\mathbb{Z}\).
    One common intuition for \(W_{n}(R)\)
    for more general \(R\) is an analogy:
    \(\mathbb{Z} / p^{n}\mathbb{Z}\) 
    is to \(\mathbb{F}_{p}\) 
    as \(W_{n}(R)\) is to \(R\).
    For rings that are not perfect,
    this intuition is less precise but still
    somewhat useful.
\end{rmk}

\begin{nota}
    We have a map (of sets) 
    \(R \xrightarrow{} W(R)\) 
    given by 
    \(c \mapsto (c, 0, \ldots)\).
    The element \((c, 0, \ldots)\) is denoted
    \([c]\).
    The map \(c \mapsto [c]\) is multiplicative
    but not additive (see \cite[Section~1]{kim-2017-witt-vec}).
    We have analogous maps \(R \xrightarrow{} W_{n}(R)\).
    By abuse of notation, we denote 
    the image of \(c\) by \([c]\).
\end{nota}

\begin{rmk}
	\label{rmk:polyraise:w2}
    Our computations will end up primarily involving
    \(W_{2}(R)\)\footnote{
    This comes from the delta formula,
    \cite[Theorem~D]{kty-2022-fedder}},
    where addition is governed by the polynomials
	\(S_{0}(X_{0}, Y_{0}) = X_{0} + Y_{0}\)
	and
	\[
		S_{1}(X_{0}, X_{1}, Y_{0}, Y_{1})
		= X_{1} + Y_{1} + 
		\frac{(X_{0} + Y_{0})^{p} - X_{0}^{p} - Y_{0}^{p}}{p}
	.\] 
	Thus we see that addition in \(W_{2}(R)\) involves raising 
	elements in \(R\) (i.e. the first component) to the \(p\)-th
	power in a lift of \(R\) to characteristic 0.
    Thus, the bottleneck for most computations with Witt vectors,
    including the present work, tends to be raising
    integer polynonmials to powers.
\end{rmk}

\subsection{Splittings of Frobenius}
\label{subsec:split:frob}

Let \(R\) be a ring of characteristic \(p\). 
We have the Frobenius morphism 
\(F \colon R \xrightarrow{} R\), 
defined as \(F(x) = x^p\).
We describe a few alternative perspectives on
the Frobenius which will be useful later.

\begin{rmk}
	\label{rmk:frob:perspectives}
    \hfill
    \begin{enumerate}[(1)]
    	\item Let \(R\) be reduced. 
    		Then we may view the Frobenius as the inclusion
    		\(R \ins R^{1 / p}\), where \(R^{1 / p}\) 
    		is the ring of formal \(p\)-th roots of elements
    		of \(R\).
    	\item Similarly, again assuming that \(R\) is reduced
    		we may view the Frobenius as the inclusion
    		\(R^{p} \ins R\).
    	\item More generally, we define \(F_{\star}R\) to
			be the \(R\)-algebra with ring structure 
			the same as \(R\), with module structure
			\(r \cdot x = F(r)x = r^{p}x\).
			Then we view the Frobenius as a map
			\(R \xrightarrow{} F_{\star}R\).
			In this description, \(F\) is an \(R\)-module 
			homomorphism as well. 
			The module \(F_{\star}R\) corresponds to the
			pushforward construction from geometry
			(i.e. pushforward of quasi-coherent
			sheaves on \(\Spec R\)).
			Even more generally, for any \(R\)-module
			\(M\) we denote \(F_{\star}M\) to be
			the analogously defined pushforward by
			Frobenius.
    \end{enumerate}
\end{rmk}

\begin{defn}
	We say that \(R\) is \textit{\(F\)-split} if the map \(F\) is
	split as a map of 
	\(R\)-modules \(R \xrightarrow{} F_{\star}R\).
\end{defn}

We will be chiefly concerned with \textit{hypersurfaces},
so we specialize to this case now.
For what follows, we assume that \(k\) is a field
of characteristic \(p > 0\) which is 
\textit{\(F\)-finite}; that is, the Frobenius
map is module-finite.

\begin{defn}
	Let \(S = k[x_{1}, \ldots, x_{n}]\).
	Then we say that \(f \in S\) is \(F\)-split if 
	\(S / (f)\) is.
\end{defn}

A fundamental fact about \(F\)-splitness is that there
exists a very concrete criterion for whether or 
not a polynomial (hypersurface) \(f\) is \(F\)-split.
First, we introduce some notation.
If \(I = (x_{1}, \ldots, x_{n})\) is a finitely generated 
ideal of some ring \(R\), 
then \(I^{[m]}\)
is defined to be \((x_{1}^{m}, \ldots, x_{n}^{m})\).

\begin{thm}
	[Fedder's Criterion]
    \label{thm:fedder:criterion}
	Let \(f \in S = k[x_{1}, \ldots, x_{n}]\).
	Let \(\mathfrak{m} = (x_{1}, \ldots, x_{n})\) 
	be the ideal generated by the variables.
	Then \(f\) is \(F\)-split if and only if 
	\(f^{p-1} \notin \mathfrak{m}^{[p]}\).
\end{thm}

\begin{proof}
	See \cite[Theorem~2.5]{ma-polstra-2021-F-sing-comm-alg}.
\end{proof}

Of particular interest is the case when \(f\) is 
homogeneous of degree \(n\), which geometrically
corresponds to a \textit{Calabi-Yau} hypersurface.
In this case, we have

\begin{thm}
	\label{thm:fsplit:ordinary}
	If \(f \in S\) 
	homogeneous of degree \(n\) 
	is \(F\)-split,
	then the Artin--Mazur height 
	of \(Z(f) = \Proj (S / (f))\) 
	is 1. That is, \(Z(f)\) is weakly ordinary.
\end{thm}

Recently, in \cite{yobuko-2019-qfs-calabi-yau}
Yobuko introduced the notion of 
quasi-\(F\)-splitness, which generalizes \(F\)-splitness.

\begin{defn}
	The ring \(R\) is \textit{\(n\)-quasi-\(F\)-split} if there exists
	a map \(\phi \colon W_{n}(R) \xrightarrow{} R\) such that
	\[
	\begin{tikzcd}
		W_{n}(R) \arrow{r}{F} \arrow{d}[swap]{} &
		F_{\star}W_{n}(R) \arrow{ld}{\phi} \\
	R 
	\end{tikzcd}
	,\]
	where the vertical map is the first Witt vector truncation.
\end{defn}

We further define the \textit{quasi-\(F\)-split height} of \(R\) as the smallest \(n\) 
for which \(R\) is \(n\)-quasi-\(F\)-split.
As above, the quasi-\(F\)-split height of 
\(f \in S = k[x_{1}, \ldots, x_{n}]\) is that
of \(R = S / (f)\).

\begin{rmk}
	Both \(F\)-splitness and quasi-\(F\)-splitness have 
	various geometric
	variants which are more general then the 
	ring-theoretic/affine versions given here. 
	These are covered extensively in the literature, for example
	see \cite{kttwyy-2022-qfs-birat}.
\end{rmk}

The quasi-\(F\)-split height also gives a generalization
of Theorem \ref{thm:fsplit:ordinary}:

\begin{thm}
	If \(f \in S\) is homogeneous of degree \(n\),
	then the Artin--Mazur height of \(Z(f)\)
	is equal to the quasi-\(F\)-split height
	of \(f\).
\end{thm}

\begin{proof}
	This is a special case of 
	\cite[Theorem~4.5]{yobuko-2019-qfs-calabi-yau}.
\end{proof}

\section{Fedder's criterion for quasi-\(F\)-splitness: an algorithmic perspective}

We now describe how the ring of Witt vectors can be used to calculate the 
quasi-\(F\)-split height (equivalently, Artin--Mazur height) of a Calabi-Yau 
hypersurface, using the version of Fedder's criterion in \cite{kty-2022-fedder}.
We will describe the algorithm in more detail. Proofs can be
found in \cite{kty-2022-fedder}.

\subsection{The computation of \(\Delta_{1}\)}

For the following discussion, 
let \(k\) be a perfect field of characteristic \(p\) and 
let \(S := k[x_{1}, \ldots, x_{n}]\).
Let \(f = \sum_{I}^{} a_{I}\mathbf{x}^{I}\) be a polynomial in \(S\).

\begin{defn}
	Let \(\Delta_{1}(f)\)
	be defined by the following equation in \(W_{2}(S)\):
	\[
		(0, \Delta_{1}(f)) = (f,0) - \sum_{I}^{} (a_{I}\mathbf{x}^{I}, 0) 
	.\] 
\end{defn}


\begin{prop}
	\label{prop:delta1:formula}
	Let \(\tilde{f}\) be a lift of \(f\) to \(W(k)\),
	i.e.
	\(\tilde{f} = \sum_{I}^{} [a_{I}] \mathbf{x}^{I}\).
	If \(k = \mathbb{F}_{p}\),
	we can compute \(\Delta_{1}(f)\) by 
	taking the reduction of
	\[
		\frac{\tilde{f}^{p} - \sum_{I}^{} ([a_{I}]x^{I})^{p} }{p}
	\] 
	mod \(p\).
\end{prop}

\begin{proof}
	Iteratively apply the formula of the first
	Witt polynomial \(S_{1}\) 
	to the monomials of \(\tilde{f}\).
\end{proof}

\begin{rmk}
	If \(f\) is a homogeneous polynomial of degree \(d\), 
	then \(\Delta_{1}(f)\) is a polynomial of degree \(pd\).
\end{rmk}

\begin{rmk}
    \label{rmk:lift:roi}
    Note that if \(k = \mathbb{F}_{p}\), we can always choose a lift
    of \(f\) with coefficients in \(\mathbb{Z}\), 
    not just \(W(k) = \mathbb{Z}_{p}\).
    In practice, we must choose a lift to the integers
    so we can represent it computationally.

    A similar fact holds for \(k = \mathbb{F}_{q}\) with 
    \(q \neq p\), replacing \(\mathbb{Z}\) with
    \(\mathbb{Z}[\zeta_{q-1}]\).
    This could be used to calculate the quasi-\(F\)-split 
    heights for surfaces not defined over \(\mathbb{F}_{p}\).
    More generally, if \(q = p^{e}\), one may take
    any abelian extension \(K\) of degree \(e\) over \(\mathbb{Q}\) such that
    \(p\) is an inert prime in \(K\).
    Then, there is a unique prime \(\mathfrak{p}\) lying over \(p\) 
    such that \(\mathcal{O}_{K} / \mathfrak{p} \isom k\).
    Furthermore, \(\mathcal{O}_{K} \ins W(k)\)
    so we may lift \(f\) to \(\mathcal{O}_{K}\) 
    and do arithmetic there. 
    Doing arithmetic in a degree \(e\) extension will have
    better complexity than in the cyclotomic extension, which
    has degree \(\phi(q) = \phi(p^{e}) = p^{e-1}(p-1)\), where
    \(\phi\) is the Euler totient function.
    We do not implement the calculation of quasi-\(F\)-split
    heights for \(q \neq p\).
\end{rmk}

Proposition \ref{prop:delta1:formula} gives a natural algorithm
for calculating the term \(\Delta_{1}(f)\):

\begin{algorithm}[H]
\label{alg:calc:delta1}
\caption{Calculation of \(\Delta_{1}(f)\) }
\KwInput{$f \in \mathbb{F}_p[x_1, \dots , x_n]$}
\KwOutput{$\Delta_1(f)$}

$\tilde{f} \gets \texttt{lift}(f)$\;
$D \gets \tilde{f}^p$\;

\For{$t \in \textnormal{\texttt{terms}} (\tilde{f})$}{
	$D \gets D - t^p$
}

$D \gets D / p$

\Return{D}
\end{algorithm}

\subsection{Splittings of Frobenius from a computational perspective}

Let \(S = k[x_{1}, \ldots, x_{n}]\) as before. 
We will use perspective (3) from
Remark \ref{rmk:frob:perspectives}; 
recall that we identify 
\(S\) with the target of Frobenius and
\(S^{p}\) with the source.
We see 
by counting degrees
that we have a generating set for \(S\) as an
\(S^{p}\)-module given by
all monomials
\(x_{1}^{i_{1}}\cdots x_{n}^{i_{n}}\)
where \(0 \leq i_{j} \leq p-1\) for all \(j\).
Moreover, since \(S\) is a polynomial ring, 
there are no (module-theoretic) relations
and
\(S\) is the free \(S^{p}\)-module generated by 
these monomials, i.e.  \[
S = \bigoplus_{1 \leq j \leq n,~ 0 \leq i_{j} \leq p-1}^{} x_{1}^{i_{1}}\cdots x_{n}^{i_{n}} S^{p}
.\] 
Then the projection of \(S\) to any of the direct sum components
is an element of \(\Hom(S,S^{p})\), which is a
splitting of Frobenius.
Let \(u\) be the projection onto the component of
\(x_{1}^{p-1}, \ldots, x_{n}^{p-1}\).

The splitting \(u\) plays an important role in \(F\)-singularity
theory, see for example 
\cite[Claim~2.6]{ma-polstra-2021-F-sing-comm-alg}.
For our purposes, we are only concerned with computing \(u\) 
for a polynomial in \(S\). 
Given \(f \in S\), we will first compute
\(u(f) \in S^{p}\), and then use the identification
\(S^{p} \isom S\) by taking \(p\)-th roots of 
exponents.

\SetKwComment{Comment}{// }{}

\begin{algorithm}[H]
\caption{Splitting of Frobenius}
\label{alg:naive:u}
\KwInput{$f \in S$}
\KwOutput{$u(f) \in S$}
$r \gets 0$\;
\For{$t \in \textnormal{\texttt{terms}} (f)$}{
	\Comment{See Def \ref{def:poly:nota} for exps() notation}
	\If{$\textnormal{\texttt{exps}} (t) \equiv (p-1, \dots, p-1) \mod p$}{
	    $\textnormal{\texttt{exps}}(t) \gets \textnormal{\texttt{exps}}(t) - (p-1, \dots, p-1)$\;
	    \Comment{This division is exact because of the previous step}
	    $\textnormal{\texttt{exps}}(t) \gets \textnormal{\texttt{exps}}(t) / p$\;
	    $r \gets r + t$
	}
}


\Return $r$
\end{algorithm}

\subsection{The naive algorithm}

Let \(f\) be a homogeneous polynomial of degree \(n\) 
in \(S\), 
so that \(Z(f)\) is a Calabi-Yau hypersurface.
Recall from Subsection \ref{subsec:split:frob} that
\(\mathfrak{m}^{[p]} = (x_{1}^{p}, \ldots, x_{n}^{p})\)
is the Frobenius power of the maximal ideal.
Following \cite{kty-2022-fedder}, we have the following
algorithm to calculate the quasi-\(F\)-split height.

\begin{algorithm}[H]
\caption{quasi-\(F\)-split height: naive algorithm}
\label{alg:qfs:naive}
\KwInput{
	Chosen bound $b \in \mathbb{N}$, \\
	~~~~~~~~~~~~~Homogeneous polynomial $f \in S$ of degree $n$
}
\KwOutput{$h(f)$}
$g \gets f^{p - 1}$\;
\If{$g \notin \mathfrak{m}^{[p]}$}{
	\Return $1$\;
}
$\Delta \gets \Delta_1(f^{p-1})$\;
$h \gets 2$\;
\While{true}{
	\If{$b < h$}{
		\Return $\infty$\;
	}
	$g \gets u(\Delta g)$\;
	\If{$g \notin \mathfrak{m}^{[p]}$}{
		\Return $h$\;
	}
	$h \gets h + 1$\;
}
\end{algorithm}

\begin{thm}
	[\cite{kty-2022-fedder}, Theorem C]
	Assume that
	\(Z(f)\) has quasi-\(F\)-split height \(h < b\).
	Then Algorithm \ref{alg:qfs:naive} terminates
	and returns \(h\).
\end{thm}

\begin{proof}
	This is just rephrasing \cite[Theorem~C]{kty-2022-fedder}.
\end{proof}

In the case of Calabi-Yau hypersurfaces, we have bounds on the height by
\cite{van-der-geer-katsura-2003-calabi-yau},
so we can deterministically recover the height.  
See also \cite[Theorem~0.1]{artin-1974-k3-surfaces},
for the case of K3 surfaces.
For a K3 surface, the height (if finite) is bounded by 10.

\begin{ex}
    Let \(p=3\), and let
    \(f = x^4 + y^4 + z^4 + w^4 \in S = k[x,y,z,w]\).
    Raising \(f\) to the \(p-1 = 2\) in \(\mathbb{F}_{3}\), we get
    \[
    g = x^8 + 2x^4y^4 + 2x^4z^4 + 2x^4w^4 + y^8 + 2y^4z^4 + 2y^4w^4 + z^8 + 2z^4w^4 + w^8
    .\] 
    The only monomial of degree \(8\) that is not contained
    in \(\mathfrak{m}^{[p]} = (x^{3}, y^{3}, x^{3}, w^{3})\) is \(x^{2}y^{2}z^{2}w^{2}\).
    Call this monomial \(c\).
    Since this coefficient of \(c\) in \(g\) is zero, we conclude
    by the classical Fedder's criterion (Theorem \ref{thm:fedder:criterion})
    that \(g\) is not \(F\)-split.
    Next, we must calculate \(\Delta_{1}(g)\).
    To do this, we take a lift of \(g\) by regarding the coefficients
    as being in \(\mathbb{Z}\) instead of \(\mathbb{F}_{3}\).
    We raise \(g^3\) in the integers, and subract by every term of \(g\) 
    cubed.

    Then, we multiply \(g\) by \(\Delta_{1}(g)\), getting a polynomial of degree 24.
    We apply the splitting of Frobenius, as described in Algorithm \ref{alg:naive:u}, 
    obtaining another polynomial of degree \(8\). 
    We can now check the term \(x^2y^2z^2w^2\) again, and it will again be zero. 
    For this choice of $f$, all of the exponents for all quantities 
    must be multiples of four.\footnote{
    In fact, this implies that the degree \(8\) polynomial is zero.
    However, this phenomenon is quite rare.
    }
    We then continue, multiplying, applying the splitting, checking the coefficient, 
    until we conclude that the height is greater than 10. 
    Now, by Hodge theory, if the height is greater than 10, it must be infinity. 
    So we conclude that the height is infinity.
\end{ex}

\subsection{The key idea: finding the matrix of the linear operator ``multiply then split''}

An implementation of Algorithm \ref{alg:qfs:naive}
is provided in MMPSingularities.jl.
The bottleneck ends up being polynomial multiplication, 
in two places:
\begin{enumerate}[(1)]
    \item raising (a lift of) \(g = f^{p-1}\) to the \(p\)-th power 
        (in line 2 of Algorithm \ref{alg:calc:delta1})
    \item multiplying \(g\) by \(\Delta_{1}(f^{p-1})\) 
        (in line 11 of Algorithm \ref{alg:qfs:naive})
\end{enumerate}

\noindent For a quartic K3 surface of characteristic \(5\),
for example, each step takes about 1 second 
using FLINT \cite{flint-2023-flint}. 

We now explain how to overcome the second
bottleneck. Since \(Z(f)\) is Calabi-Yau
(i.e. \(\deg f = n\)), we have that \(f^{p-1}\) has degree
\(n(p-1)\). 
Furthermore, by Proposition \ref{prop:delta1:formula}
the degree of \(\Delta_{1}(f^{p-1})\) 
is \(np(p - 1)\).
Thus, \(\Delta_{1}(f^{p-1})g\) has degree
\(n(p^{2} - 1)\); however, the effect 
of \(u\) on any polynomial is subtracting \(p-1\)
from the exponents of the terms and dividing by \(p\) 
(see Algorithm \ref{alg:naive:u}).
Thus, the ``multiply then split'' map 
\(g \mapsto u(\Delta_{1}(f^{p-1}) g)\) 
is a linear map from 
the space of homogeneous polynomials of degree \(n(p-1)\) 
to itself.
As a consequence of this observation,
if we can efficiently compute the matrix of 
\(g \mapsto u(\Delta_{1}(f^{p-1})g)\),
we can repeatedly apply matrix-vector multiplication.

Furthermore, when \(g\) is written as a vector 
in the basis of homogeneous monomials of degree
\(n(p-1)\), we can test if \(g \notin \mathfrak{m}^{[p]}\) 
in an especially simple way: the only 
monomial that is not in \(\mathfrak{m}^{[p]}\) is
\(x_{1}^{p-1}\cdots x_{n}^{p-1}\), 
see for example \cite{kty-2022-fedder}.
Thus, we can check if a single element of the vector representing
\(g\) is nonzero.

The algorithm for Fedder's criterion then becomes:

\begin{algorithm}[H]
\caption{Quasi-\(F\)-split height: matrix-based algorithm}
\label{alg:qfs:matrix}
\KwInput{
	Chosen bound $b \in \mathbb{N}$, \\
	~~~~~~~~~~~~~Homogeneous polynomial $f \in S$ of degree $n$
}
\KwOutput{$h(f)$}
$g \gets f^{p - 1}$\;
\If{$g \notin \mathfrak{m}^{[p]}$}{
	\Return $1$\;
}
$\Delta \gets \Delta_1(f^{p - 1})$\;
$M \gets $ the matrix of $g^{\prime} \mapsto u(\Delta g^{\prime})$\;
$h \gets 2$\;
$g_v \gets $ the representation of $g$ as a vector\;
$i \gets $ the index of the monomial $x_{1}^{p - 1} \dots x_{n}^{p - 1}$\;
\While{true}{
	\If{$b < h$}{
		\Return $\infty$\;
	}
	$g_v \gets Mg_v$\;
	\If{$g_v[i] \neq 0$}{
		\Return $h$\;
	}
	$h \gets h + 1$
}
\end{algorithm}

Thus, we have reduced our problem (algorithmically, at least)
to finding the matrix of the ``mulitply then split'' operation.

\section{Algorithms for finding the matrix of ``multiply then split''}

To find the matrix of ``multiply then split''
in the Algorithm for the quasi-\(F\)-split height
of a Calabi-Yau hypersurface, 
we must first multiply all possible 
monomials by \(\Delta_{1}(f)\) 
and then apply the map \(u\).
Here, we consider a slightly more general problem.
As usual, let \(S = k[x_{1}, \ldots, x_{n}]\).
Let $S_\ell$ denote the vector space of homogeneous degree $\ell$ polynomials in \(S\).
We let \(B_{\ell}\) denote a lexographically-ordered list
of the monomial basis of \(S_{\ell}\).
Fix some \(D \in \mathbb{N}\), and let \(\Delta \in S_D\).
Furthermore, fix \(d \in \mathbb{N}\).
We consider the problem of finding the matrix of 
\(g \mapsto u(\Delta g)\)
on the space \(S_{d}\).
The target of this map is
\(S_{d^{\prime}}\), where
\(d^{\prime} \colonequals \frac{d+D-n-1}{p}\).
By convention, if \(d^{\prime}\) is not an integer,
we declare \(S_{d^{\prime}}\) to be zero.
Note that in this case there are
there are no terms which survive \(u\), so
\(g \mapsto u(\Delta g)\) is indeed the zero map.

\begin{rmk}
    If \(\Delta = \Delta_{1}(f^{p-1})\) and \(d = n(p-1)\),
	as in the case of calculating the quasi-\(F\)-split
	height of a Calabi-Yau hypersurface, 
	then \(d^{\prime} = d\) and the matrix is 
	square.
\end{rmk}

Our first observation is that naively 
multiplying and then applying \(u\) 
as in Algorithm \ref{alg:naive:u}
does plenty of unnecessary work.
In particular, any term in the 
product \(\Delta g\) with exponents 
not congruent to 
\((p-1, \ldots, p-1) \mod p\)
is not needed. 
Even if \(g\) is a monomial, giving a linear
algorithm for multiplication, using naive
multiplication stores a large amount of unnecessary terms
in memory.
Both of our improved algorithms ignore
all of these unnecessary terms.

In what follows, 
when we apply arithmetic operations to arrays, 
(particularly addition, multiplication, and modulo)
we mean componentwise application of these operations.
Use of arbitrary componentwise operations is called
\textit{broadcasting} in Julia, and we will
sometimes use this term.
Homogenoeus polynomials are represented in computer 
memory as a list of their terms. 
This is reflected in our notation, 
thus the meaning of ``for \(\delta \in \Delta\)''
is a for loop though all the monomials
of \(\Delta\).

\begin{defn}
    \label{def:poly:nota}
    The \textit{exponent tuple} of a monomial 
    $m = ax_{1}^{d_1}, \dots, x_{n}^{d_n}$ is $(d_1, \dots, d_n)$. 
	We will denote it $\exps(m)$, and we also 
    define $\coeff(m) \colonequals a$
\end{defn}

\begin{defn}
    The weak integer compositions of $k$ into $n$ parts, 
    denoted $\wics(k, n)$, is the set of ordered tuples 
	\((d_{1}, \ldots, d_{n})\) such that 
    \(d_{1} + \ldots + d_{n} = k\).
\end{defn}

\begin{ex}
    $\wics(2, 3) = \lbrace (2, 0, 0), (0, 2, 0), (0, 0, 2), (1, 1, 0), (1, 0, 1), (0, 1, 1) \rbrace$.
\end{ex}

Thus, $\wics(d, n)$ generates the exponent tuples of the 
monomial basis of the vector space of $d$-homogeneous 
$n$-variate polynomials over a field $F$.
It follows that the dimension of the vector space is $|\wics(k, n)|$.

\begin{lem}
    \label{lem:wics:size}
    $|\wics(k, n)| = \binom{k + n - 1}{n - 1}$
\end{lem}

\begin{proof}
	This is a classical argument which goes by the 
    name ``sticks and stones,'' ``stars and bars,'' or 
    ``dots and dividers.''
\end{proof}

\begin{rmk}
    The number of terms of \(\Delta\) is bounded 
    above by \(\binom{D+n-1}{n-1}\) by Lemma \ref{lem:wics:size}.
    Likewise, the number of elements of 
    \(B_{d}\) is \(\binom{d+n-1}{n-1}\).
    For all of our algorithmic complexity calculations, 
    we let $\ell_{d} = \length(B_{d}) = \binom{n + d - 1}{n - 1}$, 
    and $\ell_{\Delta} = \length(\Delta) \leq \binom{n + D - 1}{n - 1}$.
    In the case of the quasi-\(F\)-split height of a K3 surface, 
    $\ell_{d} = \binom{n(p - 1) + n - 1}{n - 1} = \binom{np - 1}{n - 1}$, 
    and $\ell_{\Delta} = \binom{np(p - 1) + n - 1}{n - 1} = \binom{n(p^2 - p + 1) - 1}{n - 1}$.
\end{rmk}

\begin{defn}
	Let \(m\) be a monomial in \(S\). 
	Then we say that \(m\) \textit{matches}
	with another monomial \(m^{\prime}\) 
	if the monomial \(mm^{\prime}\) is
	not killed by \(u\).
\end{defn}

Concretely, this means that $\exps(m) + \exps(m^{\prime}) \equiv (p - 1, \dots, p - 1) \mod p$.

\subsection{The trivial algorithm}

\begin{algorithm}[H]
    \caption{Matrix of multiply then split: \triv}
    \label{alg:matrix:trivial}
    \KwInput{
        $\Delta, B_d, B_{d^\prime}, p$
    }
    \KwNotation{Let $\ell_d = \code{length}(B_d), \ell_{d^\prime} = \code{length}(B_{d^\prime})$}
    \KwOutput{$\ell_{d^\prime} \times \ell_d $ matrix representing ``multiply then split''}
    $M \gets \text{zero matrix of size } \ell_{d^\prime} \times \ell_d $\;
    \For{$m \in B_d$}{
        \For{$\delta \in \Delta$}{
            \Comment{If $\delta$ matches with $m$}
            \If{$\code{exps}(\delta) + \code{exps}(m) \equiv (p - 1, \dots, p - 1) \mod p$}{
                \Comment{Apply $u$ to $\delta m$}
                $\text{res} \gets (\code{exps}(\delta) + \code{exps}(m) - (p - 1, \dots, p - 1)) / p$\;
                $\text{row} \gets \code{indexof}(\text{res},B_{d^{\prime}})$\;
                $\text{col} \gets \code{indexof}(m, B_d)$\;
                $M[\text{row}, \text{col}] = \code{coeff}(\delta)$
            }
        }
    }
    \Return $M$
\end{algorithm}

In this algorithm, which we call \triv, we iterate through 
the monomials $m \in B_{d}$, each corresponding to a column in the
resulting matrix. For each monomial, we search for terms 
that match $\delta \in \Delta$, apply $u$ to their product $\delta m$, and 
get the lexographical index of the result to find which 
row to add to.

In this algorithm, the matrix can be generated in 
$O(\ell_{d}\ell_{\Delta})$ operations. 
This can be
seen from the nested loops, assuming that integer 
arithmetic operations are constant time.

In practice, the majority of the combinations of 
terms of $\Delta$ and $B_{d}$ don't match. This means in
terms of $\Delta$ and $B_{d}$ don't match. This means in
Algorithm \ref{alg:matrix:trivial}, much of our 
runtime is wasted on checking for whether terms match.
However, we emphasize that this algorithm, if 
implemented in parallel on the GPU, is indeed
fast enough to not be a bottleneck in practice.

\subsection{Modified merge-based algorithm}

We now introduce an algorithm that utilizes 
properties of monomial ordering to reduce the 
number of comparisons 
performed in checking whether two terms match.

\begin{lem}
	\label{lem:tuples:modp}
	Let \(X = (x_{1}, \ldots, x_{n}),
	Y = (y_{1}, \ldots, y_{n})\)
	be in \(\{0, \ldots, p-1\}^{n} \ins \mathbb{Z}^{n}\).
	Then \(X + Y = 
	(p-1, \ldots, p-1) \mod p\)
	if and only if 
	\(X + Y = 
	(p-1, \ldots, p-1)\).
\end{lem}

\begin{proof}
	Each coordinate has \(x_{i} + y_{i} \leq 2(p-1) = 2p-2\).
\end{proof}

\begin{cor}
    \label{cor:unique:match}
	Let $X$ be as above.
	Then there exists a unique
	match \(m(X) \in \{0, \ldots, p-1\}^{n}\) 
\end{cor}

\begin{proof}
	\(m(X) = 
	(p-1, \ldots, p-1) - X\).
	The claim follows from 
	\ref{lem:tuples:modp}.
\end{proof}

Essentially, this corollary says that
when we consider the exponent tuple mod $p$ of 
an arbitrary monomial of \(B\),
it has a unique
match mod \(p\).

\begin{cor}
	\label{cor:match:order}
	Let \(\leq_{lex}\) denote the 
	lexographical ordering.
	If \(X \leq_{lex} Y\),
	then 
	\(m(Y) \leq_{lex} m(X)\)
\end{cor}

\begin{proof}
	Follows from the definition of lexographical
	ordering.
\end{proof}

Using these two corollaries, we have the following algorithm,
which we call \merge:

\begin{algorithm}[H]
\caption{Matrix of multiply then split: \merge}
\label{alg:theta:merge}
\KwInput{
    $\Delta, B_d, B_{d^\prime}, p$
}
\KwNotation{Let $\ell_d = \code{length}(B_d), \ell_{d^\prime} = \code{length}(B_{d^\prime})$}
\KwOutput{$\ell_{d^\prime} \times \ell_d $ matrix representing ``multiply then split''}
$M \gets \text{zero matrix of size } \ell_{d^\prime} \times \ell_d $\;
$L \gets [\code{exps}(\delta) \modp ~|~ \delta \in \Delta]$\;
$R \gets [\code{exps}(m) \modp ~|~ m \in B_{d}]$\;
\Comment{Lexicographically sort $L$ and $R$ and keep the permutations}
$P_1 =$ \code{sortperm}\((L)\), $P_2 =$ \code{sortperm}\((R)\) \;
\Comment{Permute \(\Delta\) and \(B_d\)}
$\Delta^{\prime} \gets \Delta[P_1], B_d^{\prime} \gets B_d[P_2]$\;
\(l \gets 1, r \gets \code{length}(R)\)\;
\While{\(r \geq 1 ~\textnormal{and}~ l \leq \code{length}(L)\)}{
    cmp \(\gets L[l] + R[r] - (p-1, \ldots, p-1)\)\;
    \uIf(\tcp*[h]{If lex sum of pair too small, increment index of $L$}){\textnormal{cmp < 0}}{
        \(l \gets l + 1\)\;
    }\uElseIf(~\tcp*[h]{If lex sum of pair too big, decrement index of $R$}){\textnormal{cmp > 0}}{
        \(r \gets r - 1\)\;
    }\Else(~\tcp*[h]{Match found}){
        matchr \(\gets R[r]\)\;
        numLmatches \(\gets\) the number of adjacent entries in \(L\) which are equal to \(L[l]\)\;
        \Comment{For each monomial of $B_d$}
        \While{$R[r] ~\code{==}~ \textnormal{matchr} ~\textnormal{and}~ 1 \leq r$}{
            \Comment{Find matching terms of $\Delta$}
            \For{\(ll \in \{l, l+1, \ldots, l+\textnormal{numLmatches} - 1\})\)}{
                \Comment{And apply $u$}
                res \(\gets (\code{exps}(\Delta^{\prime}[ll]) + \code{exps}(B^{\prime}[r]) - (p-1, \ldots, p-1)) / p\)\;
                \(\text{col} \gets \code{indexof}(B^{\prime}[r],B_{d})\)\;
                \(\text{row} \gets \code{indexof}(\text{res},B_{d^{\prime}})\)\;
                \(M[\text{row},\text{col}] = \code{coeff}(\Delta^{\prime}[ll])\)\;
            }
            \(r \gets r - 1\)\;
        }
        \(l \gets l + \text{numLmatches} - 1\)\;
    }
}
\Return $M$
\end{algorithm}

Note that the inner while loops account for the fact that
after modding by \(p\), one does not expect 
either array to have unique elements.
The algorithm is justified by the previous corollaries.

In practice, instead of sorting the arrays in place, 
we use a method like Julia's \texttt{sortperm}
to get the sort permutation.

Because we perform two sorts, then a linear merge, 
we perform 
$\ell_{d} \log \ell_{d} + \ell_{\Delta} \log \ell_{\Delta} + \ell_{d} + \ell_{\Delta}$ operations, 
giving the algorithm complexity 
$O(\max(\ell_{d} \log \ell_{d}, \ell_{\Delta} \log \ell_{\Delta}))$.
In practice (for Calabi-Yau hypersurfaces), 
the $\ell_{\Delta} \log \ell_{\Delta}$ usually dominates.

\subsection{Weak integer compositions-based algorithm}

\begin{lem}
    \label{lem:generate:matching}
    The set of exponent tuples of \(B_{d}\) that match with $\delta$ is given by 
    \[
        \left\{ m(\exps(\delta) \mathmod p) + p\cdot w ~\Big|~ w \in \wics 
        \left(\frac{d - sum(m(\exps(\delta) \mathmod p))}{p}, n \right) \right\}
    \]
    where $sum(X)$ is the sum of the elements of tuple $X$, and $m$ is the matching function from Corollary \ref{cor:unique:match}
\end{lem}

\begin{proof}
    Direct computation from the definitions of weak integer compositions and matching terms.
\end{proof}

Because this expression is quite convoluted, we 
provide an example of a computation 
that may come up in calculating the quasi-\(F\)-split 
height of a K3 quartic surface over $\mathbb{F}_5$.

\begin{ex}
    Let $p = 5$, \(n = 4\), and $\delta$ a monomial 
    with exponent tuple $(21, 19, 22, 18)$. 
    Then $B_{16}$ is the monomial basis of 
    \(S_{16}\) (over $\mathbb{F}_5$). 
    To find the monomials of $B_{16}$ that match with 
    $\delta$, we begin by reducing $(21, 19, 22, 18) \mathmod 5 = (1, 4, 2, 3)$.
    This makes $m((1, 4, 2, 3)) = (3, 0, 2, 1)$ the 
    exponent tuple of a matching monomial mod \(p\).
    However, \((3, 0, 2, 1)\) isn't in $B_{16}$. 
    For it to be in $B_{16}$, we need to add $16 - (3 + 0 + 2 + 1) = 10$ 
    to the degree, and to maintain congruence to 
    $(4, 4, 4, 4) \mathmod 5$, we need to add two $5$'s 
    to the elements of $(3, 0, 2, 1)$. This corresponds 
    to $\{5w ~|~ w \in \wics(2, 4)\}$, which is:
    \begin{align*}
        \{&(10, 0, 0, 0), (5, 5, 0, 0), (5, 0, 5, 0), (5, 0, 0, 5), (0, 10, 0, 0), \\
        &(0, 5, 5, 0), (0, 5, 0, 5), (0, 0, 10, 0), (0, 0, 5, 5), (0, 0, 0, 10)\}
    \end{align*}
        
    \noindent Adding each of these to $(3, 0, 2, 1)$, we get the 
    exponent tuples of monomials that match with $(21, 19, 22, 18)$:
    \begin{align*}
        \{&(13, 0, 2, 1), (8, 7, 0, 1), (8, 0, 7, 1), (8, 0, 2, 6), (3, 10, 2, 1), \\
        &(3, 5, 7, 1), (3, 5, 2, 6), (3, 0, 12, 1), (3, 0, 7, 6), (3, 0, 2, 11)\}
    \end{align*}    
\end{ex}

Because we need it later for our complexity computations, we have the corollary:

\begin{cor}
    \label{cor:num:matches}
    The number of matching monomials of any term $\delta$ is 
    bounded above by $\binom{\lfloor \frac{d}{p} \rfloor + n - 1}{n - 1}$.
\end{cor}

\begin{proof}
    By Lemma \ref{lem:generate:matching}, there are 
    $\Big |\wics \left(\frac{d - sum(m(\exps(\delta) \mathmod p))}{p}, n \right) \Big |$ 
    monomials that match with $\delta$.
    Because the division is exact, 
    $\lfloor \frac{d}{p} \rfloor \geq \frac{d - sum(m(\exps(\delta) \mathmod p))}{p}$,
    giving the upper bound 
    $|\wics(\lfloor \frac{d}{p} \rfloor, n)| = \binom{\lfloor \frac{d}{p} \rfloor + n - 1}{n - 1}$.
\end{proof}

Denoting the expression from Lemma \ref{lem:generate:matching} 
by \code{generateMatchingMonomials}$(\delta, p, M_d)$, we have
the following algorithm, which we call \wicsalg:

\begin{algorithm}[H]
    \caption{Matrix of multiply then split: \wicsalg}
    \KwInput{$\Delta, B_{d}, B_{d^{\prime}}, p$}
    \KwNotation{Let \(\ell_{d} = \length(B_{d})\), 
    \(\ell_{d^{\prime}} = \length(B_{d^{\prime}})\)}
    \KwOutput{$\ell_{d^{\prime}} \times \ell_{d}$ 
    	matrix representing ``multiply then split''}
    $M \gets \text{zero matrix of size } \ell_{d^\prime} \times \ell_d $\;
    \For{$\delta \in \Delta$}{
        mons $\gets \code{generateMatchingMonomials}(\delta, p, B_{d})$\;
        \For{$m \in \textnormal{mons}$}{
            $\text{col} \gets \code{indexof}(m,B_{d})$\;
			res $\gets (\code{exps}(\delta) + \code{exps}(m) - (p-1, \ldots, p-1)) / p$\;
			$\text{row} \gets \code{indexof}(\text{res},B_{d^{\prime}})$\;
            $M[\text{row}, \text{col}] = \code{coeff}(\delta)$\;
        }
    }
    \Return $M$
    \label{alg:matrix:WICS}
\end{algorithm}

This algorithm 
generates the matrix in 
$O \left( \ell_{\Delta} \cdot \binom{\lfloor \frac{d}{p} \rfloor + n - 1}{n} \right)$ operations.
The number of matching monomials each term $\delta$ 
has is bounded above by $\binom{\lfloor \frac{d}{p} \rfloor + n - 1}{n}$ 
by Corollary \ref{cor:num:matches}, which is 
effectively constant for all reasonably computable $n$ and $p$

\subsection{Implementation}

In the implementation of \triv, \merge, and \wicsalg, there are a 
few optimizations and considerations common to all the implementations.
First, all exponent tuples are bitpacked into 
unsigned integers, improving memory efficiency 
and allowing for faster comparisons.
Note that a comparison between two unsigned integers 
is the same as a lexicographical comparison between the tuples.
Broadcasted addition and subtraction on bitpacked 
$n$-tuples reduces to addition and subtraction of 
integers, though broadcasted division and modulo still 
require $n$ separate operations.

All algorithms require the \code{indexof}$()$ function, 
which we implement using Julia's \texttt{Dict\{K,V\}}, 
which is implemented as a hashtable.

\triv ~and \wicsalg ~are easily GPU-parallelizable by creating 
a thread for each term of $\Delta$, and performing the insides of the loop in each thread.
This creates the problem of needing an \code{indexof}$()$ function, 
since Julia's \texttt{Dict\{K,V\}} is not compatible with the GPU.
The solution we implement is a static hashtable, with size 
and hashing function decided at compiletime.
Due to Julia's JIT compilation nature, creating this 
hashtable is quite slow.
Another solution can be to perform a binary search 
for the \code{indexof}$()$ function, since $B_d$ is lexicographically 
sorted, but for one-off computations, users are better 
off using non-parallelized \merge ~or \wicsalg ~to avoid 
first-time GPU kernel compilation times.
However, if many matrices need to be created, caching 
the hashtable allows the GPU implementations to far outpace 
the non-parallelized CPU implementations.

\merge ~should be parallelizable by implementing a modified parallel merge and using a parallel sort. 
However, we expect the sorting bottleneck to mean that such a parallel algorithm wouldn't beat GPU-parallel \wicsalg,
so we don't bother developing the modified parallel merge.

Below in Figure \ref{fig:momts:compare}, we compare running times 
for all the \(p\) which we consider. 
\wicsalg \, ends up being the fastest practically, 
both on the CPU and the GPU.
We did not try running the tests for \triv~ on the 
CPU for \(p=11,13\) because we do not expect 
it to finish in a reasonable
amount of time.

\begin{figure}[ht]
\label{fig:momts:compare}
\begin{center}
\begin{tabular}{c c c c c c}
    \toprule
    p & \triv ~- CPU & \triv ~- GPU & \merge   & \wicsalg ~- CPU & \wicsalg ~- GPU \\
    \midrule
    3 & 0.01   & 0.0007  & 0.001  & 0.001   & 0.0002  \\                                       
    5 & 1.6    & 0.005   & 0.046  & 0.028   & 0.0024  \\
    7 & 45.33  & 0.104   & 0.670  & 0.277   & 0.025  \\
    11 &    -  & 6.796   & 26.35  & 5.45    & 0.56  \\
    13 &    -  & 31.86   & 88.97  & 17.06   & 2.05 \\
    \bottomrule
\end{tabular}
\caption{Comparison of timings for various algorithms that calculate
the matrix of ``multiply then split''. 
Times are in seconds
and are an average of 10 different trials.}
\end{center}
\end{figure}

\section{Polynomial powering}

In this section, we describe how Algorithm \ref{alg:calc:delta1} 
is computed. Specifically, we focus on Step 2, 
\(D \gets \tilde{f}^{p}\), which is the main 
bottleneck. We first explain our implementation 
of polynomial powering using the Number Theoretic 
Transform (NTT), then explore other known polynomial 
powering algorithms and their ability to be sped up 
through mass parallelization.

\subsection{Multi-modular NTT}

\subsubsection{Kronecker substitution}
We reduce the problem of multivariate 
polynomial powering to univariate polynomial powering
by the Kronecker substitution.
Let $M$ be a non-inclusive upper bound on the degree of any variable of the polynomial
$f \in R[x_1, \dots, x_n]$. 
Then the Kronecker Substitution 
\(g(z) = f(z, z^M, z^{M^2}, \dots, z^{M^{n-1}})\)
produces a univariate polynomial $g \in R[z]$,
see \cite{arnold-2014-kronecker}.
Algebraically, this is the same as applying the 
homomorphism 
\(R[x_{1}, \ldots, x_{n}] \xrightarrow{} R[z]\)
which takes \(x_{i}\) to
\(z^{M^{i-1}}\).
This map cannot be injective in general, but it
is injective on the subset of elements of
\(f \in R[x_{1}, \ldots, x_{n}]\) 
whose terms have all degrees strictly less than \(M\).
Given two polynomials \(f_{1}, f_{2}\) that
we wish to multiply, we may choose \(M\)
big enough
so that the product lies in the subset,
so we may recover the product in 
\(R[x_{1}, \ldots, x_{n}]\) by performing
the product in the ring \(R[z]\).

\begin{rmk}
    In FLINT, exponent tuples are bitpacked into unsigned 
    integers. The process of bitpacking an array is a 
    Kronecker substitution with $M$ being a power of 2.
\end{rmk}

In the case where our polynomial is homogeneous, 
we can simply ignore one variable in all of our 
operations between monomials of the same degree. 
As above, algebraically this is applying the evaluation
homomorphism \(x_{n} \mapsto 1\).
This map is of course not injective, but it is injective on homogeneous elements.
This effectively decreases the number of variables of our polynomial by 1, 
which greatly lowers the degree of the univariate result of the
Kronecker substitution.
This improvement is essential for dense algorithms like the NTT.

\subsubsection{Polynomial powering using NTTs}
The \(k\)-th power of a univariate polynomial \(f\)
may be computed mod \(p\) by taking the NTT 
of a tuple containing the coefficients of \(f\),
raising the components to the \(k\),
and then computing the Inverse Number Theoretic Transform.
The multimodular algorithm 
computes \(f^{k}\) in the integers by
choosing primes \(p_{i}\), 
and combining the results using the Chinese remainder theorem.
Thus, our algorithm is known as the 
multi-modular NTT approach to raising polynomials to powers.

In many cases, using the Kronecker substitution to 
map a multivariate polynomial to a univariate polynomial results in 
a sparse univariate polynomial. As such, dense 
algorithms like the NTT become inefficient in both 
time and memory complexity when compared to sparse algorithms.
However, the massive throughput of GPUs allows the 
NTT to be competitive for many cases, 
quasi-\(F\)-split heights of quartic K3 surfaces being 
one of them.

\subsubsection{NTT implementation}

We use the Merge-NTT algorithm with Barrett modular 
reduction from \cite{ozcan-2023-fft} 
ported to Julia for our NTT implementation. 
So, our NTT length $L$ is always a power of $2$. 

For each NTT we perform, we search for primes $p$ that satisfy 
$p \equiv 1 \mod L$, compute primitive $L$th roots of unity 
for each $p$, and $L^{-1} \mod p$. 
These are cached, so that they can be used again for problems with the same shape 
(i.e. the degree of \(f\), the exponent \(k\), and the characteristic are the same).

Our algorithm also has fallbacks when the memory required becomes too large for the 
GPU's on-device memory.
When the problem becomes too big to fit the twiddle factors and 
inputs for all of the NTTs in device memory at once, 
we move the inputs to GPU memory and 
back to RAM for each NTT, and don't cache the twiddle factors. 
Because we are forced to move a lot of memory around for each NTT, 
and run $\log_2(L)$ modular exponentiations in each thread,
we see a large drop-off in 
performance when this NTT size threshold is reached.

\subsubsection{Prime selection}
In order for the NTT to simulate polynomial powering over $\mathbb{Z}$, 
we need to obtain an upper bound $M$ 
on the resulting coefficients, then 
choose primes $p_1 \dots p_k$ such 
that $p_1 \cdot p_2 \cdot ... \cdot p_k > M$.

The kernels and parameters from \cite{ozcan-2023-fft} are 
optimized for 64-bit integers, and when NTT sizes get big, there 
aren't enough suitable primes under $2^{32}$ to choose from.
Thus, for prime selection, 
we choose primes satisfying $p \equiv 1 \mod L$ that fit 
within 62 bits, because of the precision of Barrett Reduction.

\subsubsection{Bound finding}
The multi-modular NTT requires an upper bound $M$ in order to select primes 
to compute NTTs in. In this section, we present a quick way to compute a 
relatively tight upper bound on the resulting 
coefficients of raising a homogeneous polynomial to a power, which
allows for easier application of the NTT in polynomial powering
problems.

\begin{defn}
    Given a basis $(\mathbf{x}^{I_1}, \dots, \mathbf{x}^{I_n})$, 
    where $I_i$ are distinct degree sequences of equal 
    length, the \textit{maximal polynomial} is 
    $(m - 1)(\mathbf{x}^{I_1} + \cdots + \mathbf{x}^{I_n})$, 
    where $m$ is a non-inclusive upper bound on the coefficients.
\end{defn}

This is saying we should consider the ``worst-case'' polynomial 
for our bound-finding computations.

For small problems, we can simply plug the maximum 
polynomial $g$ into FLINT, compute $g ^ p$, and 
iterate through the resulting terms to obtain an upper bound.
We know this computation will be correct because FLINT uses GMP,
and it also must be the optimal bound on the coefficients. 
However, for larger problems, like bound finding for 
$\Delta_1$ (Algorithm \ref{alg:calc:delta1}) of
quartic K3 surfaces over $\mathbb{F}_{11}$ or $\mathbb{F}_{13}$, 
FLINT does not finish the computation in a reasonable amount of time.
If one only cares about a single shape of the problem 
(i.e. the same degree, \(k\), and characteristic),
this might be acceptable.
However, for a single computation, it completely removes 
the advantage of using the GPU, since we must perform an 
expensive CPU mutliplication to set up the algorithm.

Alternatively, we can obtain a relatively tight upper 
bound for the case of raising a homogeneous multivariate 
polynomial over $\mathbb{F}_p[x_1, \dots , x_k]$ to a power
by bounding the number of terms in the power. 

\begin{thm}
    Let $f \in \mathbb{Z}[x_1, \dots, x_n]$ be homogeneous of degree \(h\), 
    and let $m$ be a non-inclusive upper bound on the coefficients. 
    Then, the coefficients of $f ^ k$ are bounded above by 
    $\left((m - 1) \cdot \binom{h + n - 1}{n - 1}\right)^ k$
\end{thm}

\begin{proof}
    We induct on $k$. The maximum coefficient 
    of $f^0$ is $1$, which is bounded above by 
    
    \noindent$\left((m - 1) \cdot \binom{h + n - 1}{n - 1}\right)^ 0 = 1$.
    Let $(d_1, \dots , d_n)$ denote the exponent 
    tuple of a term of $f^k$, let $(d'_1, \dots , d'_n)$ 
    denote the exponent tuple of a term of 
    $f^{k + 1}$, and let $(a_1, \dots , a_n)$ 
    denote the exponent tuple of a term of $f$.

    Consider an arbitrary term of $f^{k + 1}$. 
    To find all terms in the unreduced 
    expansion of $f^k \cdot g$ that contribute 
    to that term, we look at 
    terms of $f^k$ with degree sequences of 
    the form $(d'_1 - a_1, \dots , d'_n - a_n)$. 
    The number of these terms of $f^k$ is bounded 
    above by $|wics(h, n)|$, or $\binom{h + n - 1}{n - 1}$ 
    by Lemma \ref{lem:wics:size}. Using our 
    inductive assumption, the maximum coefficient 
    of $f^k$ is bounded above by 
    $\left((m - 1) \cdot \binom{h + n - 1}{n - 1}\right)^ k$, 
    so multiplying each of these by coefficients 
    of $f$, which have a maximum value of 
    $m - 1$, and adding up $\binom{h + n - 1}{n - 1}$ 
    copies of these gives 
    $\left((m - 1) \cdot \binom{h + n - 1}{n - 1}\right)^{k + 1}$.
\end{proof}

To apply this formula to computing $\Delta_1(f^{p - 1})$ for K3 
quartic surfaces over $\mathbb{F}_p$, we plug in $n = 4$, $h = 4p$, $m = p$, 
and $k = p$ to obtain $M = \left((p - 1) \cdot \binom{4p + 3}{3}\right)^ p$
for an upper bound. To obtain the optimal upper bound, we can plug in the
maximal polynomial into a multimodular NTT with primes multiplying to over $M$,
and retrieve the maximum coefficient of that result.

\subsection{Other algorithms} \label{sec:poly:other}

Many other algorithms 
for polynomial powering
are described in \cite{monagan-2012-sparse-powering}. 
They compare performance in sparse and dense cases. 
Here, we breifly comment on how these algorithms perform on the GPU.
We refer the reader to \cite{monagan-2012-sparse-powering}
for a more precise description of all these algorithms.
In the following, assume we have some polynomial \(f\) which we wish 
to raise to the \(k\)-th power.

\code{RMUL} is analogous to the classic FOIL algorithm which is taught in schools. 
This is bottlenecked (at least on the GPU) by the ``collect like terms'' step, 
which requires a parallel sort.
While there are GPU-optimized sorting algorithms, such as the parallel merge sort
provided by CUDA.jl, it is a relatively expensive operation on the GPU.
There is a sorting algorithm proposed in \cite{gupta-2023-gpu-sort}
which claims to beat the state-of-the-art, but we were unable to reproduce the result.

\code{RSQR} uses the binary expansion of \(k\) to find the power in less
total multiplications than \code{RMUL}.
\code{BINA} and \code{BINB} use binomial expansion to more efficiently 
expand \(f^{k}\), using \code{RMUL} to merge at the end.
They both perform better in the case that the problem is sparse.
On the GPU, all of these are bottlenecked by sorting, just like
\code{RMUL}.

\code{MNE} uses multinomial coefficients to expand \(f^{k}\) 
and then combines like terms with a sort. 
Unlike the previous four algorithms, it is not bottlenecked by the sort;
instead, it is bottlenecked by the memory required to store the table of
multinomial coefficients. Do note that \code{MNE} is 
competitive for small problems.

\code{SUMS} and \code{FPS} have dense and sparse versions which are described in 
\cite{monagan-2012-sparse-powering}.
In particular, \code{FPS} is implemented in FLINT, and is called by our
code when we need powering on the CPU.
They can be parallelized, as discussed in \cite{monagan-2012-sparse-powering},
but this requires the use of locks and a heap data structure,
which are more challenging to implement on the GPU.
It would be interesting to have a GPU-accelerated version of
\code{FPS} and compare its performance with 
multimodular NTT.
The authors expect one could get a big improvement in the sparse case.

\subsection{Evaluation}

To demonstrate the power of using the GPU for mathematical computations, 
we compare our implementation
with two existing computational mathematics libraries, FLINT
and MAGMA.
We take a homogeneous polynomial of degree 16 in the integers,
whose coefficients are randomly chosen in the set 
\(\{0, \ldots, 4\}\), and raise it to successive powers \(n\),
starting at \(n=5\).
Raising such a polynomial to the 5th power is a similar 
computation to the bottleneck in the calculation of the
quasi-\(F\)-split height of a quartic surface over \(\mathbb{F}_{5}\).

\begin{figure}[h]
\begin{center}
\begin{tabular}{c c c c}
    \toprule
    \(n\) & MAGMA & FLINT & GPUPolynomials.jl \\
    \midrule
    5 & 2.10 &    0.80 & 0.001 \\
    6 & 5.82 &    1.64 & 0.003 \\
    7 & 12.03 &   2.93 & 0.005 \\
    8 & 21.64 &   5.99 & 0.010 \\
    9 & 38.78 &   12.5 & 0.011 \\
    10 & 63.22 &  20.7 & 0.012 \\
    11 & 95.48 &  29.8 & 0.038 \\
    12 & 137.61 & 40.3 & 0.043 \\
    13 & 194.73 & 52.2 & 0.074 \\
    14 & 267.57 & 70.0 & 0.079 \\
    15 & 359.22 & 90.1 & 0.087 \\
    \bottomrule
\end{tabular}
\caption{Polynomial powering times (in seconds) for various powers}
\end{center}
\end{figure}

Note that the above numbers are really the GPU beating the CPU.
The NTT is a dense algorithm, while FLINT uses a sparse algorithm.
We do not know which algorithm MAGMA uses for polynomial powering
in the integers.
In theory, such a sparse algorithm should
be much more efficient.
However, the throughput of the GPU is so much better that the
GPU crushes the CPU, even with a worse algorithm.
The authors hope that these numbers can inspire others who
rely on fundamental algorithms such as those in FLINT to 
consider re-implementing them on the GPU.

Additionally note that this does not mean the GPU will give 
improvements for every polynomial powering 
problem. Generally, the NTT is powerful when the resulting 
coefficients are small, and the polynomial is reasonably dense. 
In more sparse problems, as discussed before in section \ref{sec:poly:other}, 
the NTT is less effective, and other algorithms benefit less 
from parallelization, making them more 
suited for the CPU instead.

\section{Matrix multiplication mod \(p\)}

In the last few years, a lot of work has gone into optimizing matrix multiplication,
especially with floating-point data types
(for example, \cite{nvidia-2024-cublas}, \cite{openblas-2024-openblas}). 
Today, floating point types are faster than
integer data types.
For example, on Nvidia devices, Float32 matrix multiplication
is twice as fast as Int32 matrix multiplication, and Float64
multiplication is supported in hardware while Int64 multiplication
is not.
Moreover, the IEEE standards guarantee that integer multiplication (i.e. those that only use the mantissa) 
is guaranteed to be exact even in floating point types. 
Thus, we can freely treat floating points as an integer type of smaller size.
This trick has been utilized many times in the literature (e.g. see \cite{bglm-2024-matmul-modp}).
Here, we implement a simple version of using this trick in Julia,
which suffices for our purposes.

Let \(\ell\) be the largest possible value of an integer for our data type.
Each entry ranges from \(0\) to \(N-1\).
Thus, the maximum number of operations $o$ before our datatype overflows is one less than the value $o'$ such that 
$o' \cdot (N^2 - 2N + 1) = o^{\prime} \cdot (N-1)^{2} $ is larger than the integer limit $\ell$. Thus
\begin{align*}
    o = \left\lfloor \frac{\ell}{(N^2 - 2N + 1)} \right\rfloor - 1
\end{align*}

For Nvidia GPUs, we wish to use the Float32 type, for the primes
\(p \in \{3,5,7,11,13\}\).
Respectively, for each of these values we have
\(o = 4194302, 599185, 246722, 135299, 85597, 59073\).
Furthermore, the sizes of the matrices in question are 
respectively
 \(165, 969, 2925, 12341, 20825\).
Thus, we see that for \(p \leq 13\), we can use a fully floating-point
library like CUBLAS naively.
\footnote{
    Note that we use OpenBLAS and CUBLAS; 
    we are really using Julia wrappers
    provided by libraries such as Oscar, 
    CUDA.jl, or the Julia standard library.
    Note that MAGMA also provides
    a mod $p$ matrix multiplication implementation,
    which (according to its documentation) wraps CUBLAS for 
    for p=11 and
    p=13. 
}
To support larger primes, we implement a 
simple GPU-based fallback implementation of matrix multiplication using CUDA.jl
based on \cite{mao-2024-matmul}
which reduces mod \(p\) every 32 entries.
Our implementation is provided in GPUFiniteFieldMatrices.jl.
While it doesn't come close to CUBLAS, and is in fact slower
than a OpenBLAS when using CPU multithreading, it gives exact
computations for arbitrarily large
matrices (as long as \(32 < o\))
and is good enough that
matrix multiplication won't be a bottleneck in 
Fedder type criterion calculations.

To illustrate the performance difference between CPU and GPU, we timed
multiplying matrices of various sizes on the CPU and GPU in 
characteristic 11 (Figure 3).

\begin{figure}[h]
\begin{center}
\begin{tabular}{c c c c c}
    \toprule
    Size & CPU single-threaded & CPU multi-threaded & CUBLAS & Fallback implementation \\
    \midrule
    5,000  & 2.47   & 1.18  & 0.02 & 1.78   \\
    10,000 & 18.72  & 8.03  & 0.13 & 15.00  \\
    15,000 & 62.90  & 25.30 & 0.43 & 54.95  \\
    20,000 & 150.76 & 58.83 & 3.66 & 114.72 \\
    \bottomrule
\end{tabular}
\caption{Matrix multiplication times (in seconds) for various sizes of matrices}
\end{center}
\end{figure}

\section{Heights of K3 surfaces}
\label{sec:heights:surfaces}

\subsection{Recollections on the moduli of K3 surfaces}


\begin{thm}
    [Lang-Weil, \cite{lang-weil-1954-estimate} Theorem 1]
	Let \(X \ins \mathbb{P}^{n}\)
	be a projective variety of dimension \(r\). 
	Then 
	\[
		\#X(\mathbb{F}_{p}) = p^r + O(p^{r - \frac{1}{2}})
	\] 
\end{thm}

Say we have an ambient projective variety
\(Y\) with chosen
hypersurface \(D\).
The Lang-Weil estimate roughly
says that if we pick a random point \(x\),
we can expect \(x\) to lie in \(D\) 
with probably about \(1 / p\).
By \cite[Section~7]{artin-1974-k3-surfaces}, 
the locus \(M_{i}\) of height \(i\) such that \(h \leq i\)
is cut out by a single section in 
\(M_{i-1}\).
If we apply both of these facts,
with \(Y\) being the 
moduli space\footnote{
	say, the course moduli space of polarized K3 surfaces,
	though one should be able to make a similar statement
	for the moduli stack
} of quartic K3 surfaces,
we can conclude 
the probability of a random point 
in the moduli space being in the locus
\(M_{2}\) is about \(1 / p\). 
Inductively, we see 

\begin{heur}
The probability of a randomly
chosen surface having height \(h\) is about \(1 / p^{h}\).
\end{heur}

Thus, we may find a surface of height \(h\) by 
randomly choosing quartic K3 surfaces, and 
if we compute a few times \(p^{h}\) 
samples, we can be confident that we'll find
one with high probability.

In practice, our methodology is to choose
random quartic polynomials,
by sampling a point in the vector space
of homogeneous polynomials, which is isomorphic to
\(\mathbb{F}_{p}^{35}\).
To obtain the moduli space \(M\) of quartic K3 surfaces, 
we must projectivize and take the quotient by the action of 
\(\text{PGL}_{4}(\mathbb{F}_{p})\) which acts by changes of variables. 
Thus, there may be two sources of deviation from the expected
probability--the group quotient and the actual 
error term in the estimate.

\subsection{Computations}

Over \(\mathbb{F}_{5}\), we 
found that the probability of finding a K3 surface
of height \(h\) was about \(1 / 5^{h}\), to three digits
of precision, for all heights \(h \leq 6\).
For higher heights we had less samples
and more variance, although
the probabilities seem to be less than expected
for higher heights.
Similarly, over \(\mathbb{F}_{7}\) we found that
the probabilities very closely matched \(1 / 7^{h}\) 
for low heights, with a more variance at higher heights.

For \(p=5\), the algorithm throughput is about
1400 surfaces per second on 
Nvidia 2080Ti GPUs provided by the 
UCSD research cluster.
Since much less time is taken for height \(1\) examples
(since the classical Fedder's criterion suffices), 
for the purposes of estimating the time to compute a
high height example we may ignore them.
Thus,
the expected compute time necessary to find a height = \(10\) 
example is about  \(5^{9} / 1400 \approx 1395\) seconds, or about 
23 compute minutes.
For \(p=7\), the throughput is about 185 surfaces
per second.
Thus, the expected compute time to compute
a height = \(10\) example is 
\(7^{9} / 185 \approx \) 218,127 seconds, or about 
60 compute hours.
The actual times to find the examples were much 
longer than this, because the authors were using a 
less-optimized NTT and had not yet discovered $\wicsalg$
and were using \(\merge\) instead.

For $p = 11$ and $p = 13$, the threshold for the NTT where 
GPU memory becomes a bottleneck is reached, and the algorithm 
faces a sharp drop in speed. The examples below took about 12 compute hours
on the aforementioned 2080Tis (via random guessing).
Newton polygons of K3 surfaces may be computed by using the
library ToricControlledReduction
\cite{chk-2019-toric-controlled-reduction} 
(which computes the zeta
function of the surface, from which the Newton polygon and thus the
height can be deduced).
Our method is comparable to ToricControlledReduction to find
the height for \(p = 11\). On an Nvidia 3090 our method took 
around 39-40 seconds per surface, while ToricControlledReduction took 
about 43 seconds per surface.
For \(p = 13\), our 
code took 100-110 seconds per surface, while ToricControlledReduction
again took about 43 seconds per surface.




\begin{figure}[htbp]
	\begin{center}
        \def\arraystretch{1.5}
		\begin{tabular}{p{0.1\linewidth}p{0.8\linewidth}}
			 \toprule
             \textbf{Height} & \textbf{Equation} \\
			 \midrule
			 1 & \(x_{1}^{4} + x_{2}^{4} + x_{3}^{4} + x_{4}^{4}\) \\
			  
			 2 & \(4 x_1^4 + 2 x_1^3 x_2 + x_1^3 x_4 + 4 x_1^2 x_2^2 + 2 x_1^2 x_2 x_3 + 2 x_1^2 x_3^2 + x_1^2 x_3 x_4 + 3 x_1 x_2^3 + 4 x_1 x_2^2 x_3 + 4 x_1 x_2^2 x_4 + 2 x_1 x_2 x_3 x_4 + 3 x_1 x_2 x_4^2 + 3 x_1 x_3^3 + x_1 x_3^2 x_4 + x_1 x_3 x_4^2 + x_1 x_4^3 + 4 x_2^4 + 2 x_2^3 x_3 + 4 x_2^3 x_4 + 4 x_2^2 x_3^2 + x_2^2 x_3 x_4 + 2 x_2^2 x_4^2 + 3 x_2 x_3^3 + 4 x_2 x_3^2 x_4 + 4 x_2 x_3 x_4^2 + 2 x_2 x_4^3 + 2 x_3^4 + 2 x_3^3 x_4 + 2 x_3^2 x_4^2 + x_3 x_4^3 + 4 x_4^4\) \\
			  
			 3 & \(2 x_1^4 + x_1^3 x_2 + 3 x_1^3 x_3 + x_1^3 x_4 + x_1^2 x_2 x_3 + 4 x_1^2 x_2 x_4 + x_1^2 x_3^2 + 4 x_1^2 x_3 x_4 + 3 x_1^2 x_4^2 + 4 x_1 x_2^3 + 3 x_1 x_2^2 x_3 + x_1 x_2^2 x_4 + 2 x_1 x_2 x_3^2 + 3 x_1 x_2 x_3 x_4 + x_1 x_3^3 + 4 x_1 x_3 x_4^2 + 2 x_1 x_4^3 + x_2^3 x_3 + 3 x_2^3 x_4 + 4 x_2^2 x_3^2 + 4 x_2^2 x_3 x_4 + x_2^2 x_4^2 + 2 x_2 x_3^3 + 3 x_2 x_3^2 x_4 + 4 x_2 x_3 x_4^2 + 3 x_2 x_4^3 + 4 x_3^4 + 3 x_3^3 x_4 + 2 x_3 x_4^3 + 3 x_4^4\) \\
			  
			 4 & \(4 x_1^4 + 2 x_1^3 x_3 + 4 x_1^3 x_4 + 3 x_1^2 x_2^2 + 3 x_1^2 x_2 x_3 + 4 x_1^2 x_2 x_4 + 2 x_1^2 x_3 x_4 + x_1^2 x_4^2 + 3 x_1 x_2^3 + x_1 x_2^2 x_3 + x_1 x_2^2 x_4 + x_1 x_2 x_3^2 + x_1 x_2 x_3 x_4 + x_1 x_2 x_4^2 + 2 x_1 x_3^3 + 2 x_1 x_3^2 x_4 + x_1 x_3 x_4^2 + 2 x_1 x_4^3 + 4 x_2^4 + 3 x_2^3 x_3 + x_2^3 x_4 + 3 x_2^2 x_3^2 + 3 x_2^2 x_3 x_4 + x_2^2 x_4^2 + 2 x_2 x_3^3 + 3 x_2 x_3^2 x_4 + x_2 x_3 x_4^2 + 3 x_2 x_4^3 + 3 x_3^4 + 2 x_3^3 x_4 + 4 x_3^2 x_4^2 + x_3 x_4^3\) \\
			  
			 5 & \(2x_1^4 + 2x_1^3x_2 + 4x_1^3x_3 + 3x_1^3x_4 + 2x_1^2x_2^2 + 4x_1^2x_2x_3 + x_1^2x_2x_4 + 2x_1^2x_3^2 + 3x_1^2x_3x_4 + 4x_1x_2^2x_3 + 3x_1x_2^2x_4 + x_1x_2x_3^2 + x_1x_2x_3x_4 + 2x_1x_2x_4^2 + 3x_1x_3^3 + 3x_1x_3^2x_4 + x_1x_3x_4^2 + x_2^4 + x_2^3x_3 + 2x_2^2x_3^2 + 2x_2^2x_3x_4 + 3x_2^2x_4^2 + 2x_2x_3^3 + 2x_2x_3^2x_4 + 2x_2x_3x_4^2 + 4x_3^4 + x_3^3x_4 + 2x_3^2x_4^2 + 3x_4^4\) \\
			  
			 6 & \(x_1^3x_2 + x_1^3x_3 + 3x_1^3x_4 + 3x_1^2x_2^2 + 2x_1^2x_2x_4 + 4x_1^2x_3x_4 + 4x_1^2x_4^2 + x_1x_2^3 + 3x_1x_2^2x_3 + 4x_1x_2^2x_4 + 2x_1x_2x_3^2 + 2x_1x_2x_3x_4 + 2x_1x_2x_4^2 + 2x_1x_3^3 + 3x_1x_3^2x_4 + 3x_1x_3x_4^2 + x_1x_4^3 + 4x_2^4 + x_2^3x_3 + x_2^3x_4 + x_2^2x_3^2 + 4x_2^2x_3x_4 + x_2^2x_4^2 + 3x_2x_3^3 + 2x_2x_3^2x_4 + 3x_2x_4^3 + 4x_3^4 + x_3^3x_4 + 3x_3x_4^3 + x_4^4\) \\
			  
			 7 & \(4x_1^4 + x_1^3x_3 + 3x_1^3x_4 + 4x_1^2x_2^2 + 2x_1^2x_2x_3 + 2x_1^2x_2x_4 + 2x_1^2x_3^2 + 4x_1^2x_3x_4 + 4x_1x_2^2x_3 + 2x_1x_2x_3^2 + x_1x_2x_4^2 + 2x_1x_3^3 + 4x_1x_3^2x_4 + 2x_1x_3x_4^2 + x_1x_4^3 + 4x_2^4 + 3x_2^3x_4 + 3x_2^2x_3^2 + x_2^2x_3x_4 + 2x_2^2x_4^2 + 3x_2x_3^2x_4 + 4x_2x_3x_4^2 + 3x_2x_4^3 + 3x_3^3x_4 + x_3^2x_4^2 + 4x_4^4\) \\
			  
			 8 & \(x_1^4 + 2x_1^3x_2 + 4x_1^3x_3 + x_1^2x_2^2 + 4x_1^2x_2x_3 + x_1^2x_2x_4 + x_1^2x_3x_4 + 2x_1x_2^2x_3 + 2x_1x_2^2x_4 + 2x_1x_2x_3^2 + 4x_1x_2x_3x_4 + 3x_1x_2x_4^2 + 3x_1x_3^3 + 4x_1x_3^2x_4 + 3x_1x_3x_4^2 + x_1x_4^3 + 4x_2^4 + 4x_2^3x_3 + x_2^3x_4 + 4x_2^2x_3^2 + 2x_2^2x_3x_4 + x_2^2x_4^2 + 4x_2x_3^2x_4 + x_2x_4^3 + x_3^4 + 2x_3^3x_4 + x_3^2x_4^2 + 4x_4^4\) \\
			  
			 9 & \(3 x_1^4 + 3 x_1^3 x_2 + 3 x_1^3 x_3 + x_1^2 x_2^2 + 3 x_1^2 x_2 x_3 + 3 x_1^2 x_2 x_4 + 3 x_1^2 x_3^2 + 2 x_1^2 x_3 x_4 + 2 x_1^2 x_4^2 + 4 x_1 x_2^3 + 2 x_1 x_2^2 x_3 + 4 x_1 x_2 x_3^2 + 2 x_1 x_2 x_3 x_4 + 4 x_1 x_2 x_4^2 + x_1 x_3^3 + 3 x_1 x_3^2 x_4 + 3 x_1 x_3 x_4^2 + x_1 x_4^3 + 3 x_2^3 x_3 + 4 x_2^3 x_4 + 3 x_2^2 x_3 x_4 + x_2^2 x_4^2 + 4 x_2 x_3^2 x_4 + 4 x_2 x_3 x_4^2 + 4 x_2 x_4^3 + 3 x_3 x_4^3 + 4 x_4^4\) \\
			  
			 10 & \(2 x_1^4 + 4 x_1^3 x_2 + 3 x_1^3 x_3 + x_1^3 x_4 + x_1^2 x_2^2 + 2 x_1^2 x_2 x_3 + 2 x_1^2 x_2 x_4 + 4 x_1^2 x_3^2 + 4 x_1^2 x_3 x_4 + 2 x_1^2 x_4^2 + x_1 x_2^3 + 4 x_1 x_2^2 x_4 + 3 x_1 x_2 x_3^2 + 3 x_1 x_2 x_4^2 + 2 x_1 x_3^3 + 3 x_1 x_3^2 x_4 + 2 x_1 x_3 x_4^2 + x_1 x_4^3 + 3 x_2^4 + 2 x_2^3 x_3 + 2 x_2^3 x_4 + 4 x_2^2 x_3^2 + 3 x_2^2 x_3 x_4 + 3 x_2^2 x_4^2 + x_2 x_3^3 + 2 x_2 x_3 x_4^2 + 2 x_2 x_4^3 + 4 x_3^4 + x_3^3 x_4 + 3 x_3^2 x_4^2 + 4 x_3 x_4^3 + 3 x_4^4\)\\
			 
			 \(\infty\)& \(x_{1}^{4} + x_{2}^{4} + x_{3}^{4} + x_{4}^{4} + x y z w\) \\
             \bottomrule
			  
		\end{tabular}
	\end{center}
	\caption{Quartic K3 surfaces with specified Artin--Mazur height over \(\mathbb{F}_{5}\)}
\end{figure}

\begin{figure}[htbp]
	\begin{center}
		\def\arraystretch{1.5}
		\begin{tabular}{p{0.1\linewidth}p{0.8\linewidth}}
			 \toprule
             \textbf{Height} & \textbf{Equation} \\
			 \midrule
			 
			 1 & \(5x_1^4 + 5x_1^3x_3 + 2x_1^3x_4 + 3x_1^2x_2x_3 + x_1^2x_2x_4 + 6x_1^2x_3^2 + 3x_1^2x_3x_4 +
                 3x_1^2x_4^2 + 4x_1x_2^3 + 6x_1x_2^2x_3 + 2x_1x_2^2x_4 + 4x_1x_2x_3^2 + 5x_1x_2x_3x_4 + 4x_1x_2x_4^2
                 + 5x_1x_3^3 + 4x_1x_3^2x_4 + 5x_1x_4^3 + 5x_2^4 + x_2^3x_3 + 4x_2^3x_4 + 5x_2^2x_3^2 + x_2^2x_3x_4 +
                 x_2x_3^3 + 2x_2x_3^2x_4 + 2x_2x_3x_4^2 + x_2x_4^3 + 3x_3^4 + 5x_3^3x_4 + 3x_3^2x_4^2 + x_4^4 \) \\
			 
             2 & \(3x_1^4 + 4x_1^3x_2 + x_1^3x_3 + x_1^3x_4 + x_1^2x_2^2 + 5x_1^2x_2x_3 + 5x_1^2x_2x_4 + 
             3x_1^2x_3^2 + 5x_1^2x_3x_4 + 6x_1^2x_4^2 + 2x_1x_2^3 + x_1x_2^2x_3 + 5x_1x_2^2x_4 + 2x_1x_2x_3x_4 + 
             x_1x_2x_4^2 + 2x_1x_3^3 + 3x_1x_3^2x_4 + x_1x_3x_4^2 + x_1x_4^3 + 4x_2^4 + 4x_2^3x_3 + 4x_2^3x_4 + 
             6x_2^2x_3^2 + 3x_2^2x_3x_4 + 3x_2x_3^3 + 4x_2x_3x_4^2 + 2x_3^4 + 4x_3^3x_4 + 4x_3^2x_4^2 + 2x_3x_4^3 + 6x_4^4\) \\
			 
			 3 &  \(4x_1^4 + x_1^3x_2 + 2x_1^3x_3 + 6x_1^3x_4 + 6x_1^2x_2^2 + 3x_1^2x_2x_3 +
             3x_1^2x_2x_4 + 2x_1^2x_3x_4 + 4x_1^2x_4^2 + 2x_1x_2^3 + 5x_1x_2^2x_4 + 5x_1x_2x_3^2 +
             4x_1x_2x_3x_4 + 4x_1x_2x_4^2 + 6x_1x_3^3 + x_1x_3^2x_4 + 5x_1x_3x_4^2 + 2x_1x_4^3 +
             3x_2^4 + 2x_2^3x_3 + 5x_2^2x_3^2 + 5x_2^2x_3x_4 + 3x_2^2x_4^2 + 4x_2x_3^3 +
             6x_2x_3^2x_4 + 5x_2x_3x_4^2 + 3x_2x_4^3 + 4x_3^3x_4 + 4x_3^2x_4^2 + x_3x_4^3 + 5x_4^4\)
             \\
			 
			 4 & \(2x_1^4 + 6x_1^3x_2 + 3x_1^3x_3 + x_1^3x_4 + 4x_1^2x_2^2 + 3x_1^2x_2x_3 +
             3x_1^2x_2x_4 + 2x_1^2x_3^2 + x_1^2x_3x_4 + 2x_1^2x_4^2 + 3x_1x_2^3 + 6x_1x_2^2x_4 +
             x_1x_2x_3^2 + 6x_1x_2x_3x_4 + x_1x_2x_4^2 + 4x_1x_3^3 + 2x_1x_3^2x_4 + 5x_1x_3x_4^2 +
             2x_1x_4^3 + 6x_2^4 + 3x_2^3x_3 + 5x_2^2x_3^2 + x_2^2x_3x_4 + 5x_2^2x_4^2 + 4x_2x_3^3 +
             3x_2x_3^2x_4 + x_2x_4^3 + 6x_3^4 + 2x_3^3x_4 + x_3^2x_4^2 + 3x_3x_4^3 + 2x_4^4\) \\
			 
			 5 & \(5x_1^4 + 6x_1^3x_2 + 2x_1^3x_3 + 3x_1^3x_4 + 4x_1^2x_2^2 + 3x_1^2x_2x_4 +
             2x_1^2x_3^2 + 3x_1^2x_3x_4 + 6x_1^2x_4^2 + 4x_1x_2^2x_3 + 6x_1x_2^2x_4 + 2x_1x_2x_3^2 +
             3x_1x_2x_3x_4 + 5x_1x_2x_4^2 + 3x_1x_3^3 + x_1x_3^2x_4 + 5x_1x_3x_4^2 + 6x_2^4 +
             5x_2^3x_3 + 3x_2^2x_3^2 + 6x_2^2x_3x_4 + 3x_2x_3^3 + 3x_2x_3^2x_4 + 4x_2x_3x_4^2 +
             3x_2x_4^3 + 5x_3^4 + 6x_3^2x_4^2 + 6x_3x_4^3 + 3x_4^4\) \\
			 
			 6 & \(x_1^4 + x_1^3x_2 + 4x_1^3x_3 + 6x_1^3x_4 + 6x_1^2x_2^2 + 2x_1^2x_2x_4 +
             6x_1^2x_3x_4 + 6x_1^2x_4^2 + 4x_1x_2^3 + 3x_1x_2^2x_3 + 2x_1x_2^2x_4 + 2x_1x_2x_3^2 +
             5x_1x_2x_3x_4 + 6x_1x_2x_4^2 + 6x_1x_3^2x_4 + 3x_1x_3x_4^2 + 6x_2^4 + 2x_2^3x_3 +
             3x_2^3x_4 + 5x_2^2x_3^2 + 4x_2^2x_3x_4 + 6x_2^2x_4^2 + 5x_2x_3^2x_4 + x_2x_3x_4^2 +
             3x_2x_4^3 + 2x_3^4 + 2x_3^3x_4 + 5x_3^2x_4^2 + 2x_3x_4^3 + 4x_4^4 \) \\
			 
			 7 & \(2x_1^3x_2 + 2x_1^3x_3 + 2x_1^3x_4 + x_1^2x_2^2 + 2x_1^2x_2x_3 + 3x_1^2x_2x_4 +
             5x_1^2x_3^2 + 6x_1^2x_3x_4 + x_1^2x_4^2 + 2x_1x_2^3 + 5x_1x_2^2x_3 + x_1x_2x_3^2 +
             2x_1x_2x_3x_4 + 6x_1x_2x_4^2 + 4x_1x_3^3 + 6x_1x_3^2x_4 + 5x_1x_3x_4^2 + 2x_1x_4^3 +
             2x_2^3x_3 + 3x_2^3x_4 + 4x_2^2x_3^2 + 3x_2^2x_4^2 + 3x_2x_3^3 + x_2x_3^2x_4 +
             5x_2x_3x_4^2 + 5x_2x_4^3 + 5x_3^3x_4 + x_3^2x_4^2 + 6x_3x_4^3 + 6x_4^4\) \\
			 
			 8 & \(2x_1^3x_2 + 2x_1^3x_4 + 4x_1^2x_2^2 + 6x_1^2x_2x_3 + 5x_1^2x_2x_4 + 4x_1^2x_3^2 +
             3x_1^2x_3x_4 + 3x_1^2x_4^2 + 4x_1x_2^3 + x_1x_2^2x_3 + x_1x_2^2x_4 + 4x_1x_2x_3^2 +
             5x_1x_2x_3x_4 + x_1x_2x_4^2 + 3x_1x_3^3 + x_1x_3^2x_4 + 3x_1x_3x_4^2 + x_1x_4^3 +
             5x_2^3x_3 + 5x_2^3x_4 + 6x_2^2x_3x_4 + 6x_2^2x_4^2 + 4x_2x_3^2x_4 + 3x_2x_3x_4^2 +
             2x_2x_4^3 + 6x_3^3x_4 + 6x_3^2x_4^2 + 4x_3x_4^3\) \\
			 
			 9 & \(2x_1^3x_2 + x_1^3x_3 + 6x_1^3x_4 + 6x_1^2x_2^2 + 4x_1^2x_2x_3 + 2x_1^2x_2x_4 +
             3x_1^2x_3x_4 + x_1^2x_4^2 + x_1x_2^3 + x_1x_2^2x_3 + 6x_1x_2^2x_4 + 6x_1x_2x_3^2 +
             6x_1x_2x_3x_4 + 6x_1x_2x_4^2 + 2x_1x_3^3 + 4x_1x_3x_4^2 + 6x_1x_4^3 + 6x_2^3x_3 +
             4x_2^3x_4 + 3x_2^2x_3^2 + 4x_2x_3^3 + 5x_2x_3^2x_4 + 4x_2x_3x_4^2 + 5x_2x_4^3 +
             3x_3^3x_4 + 4x_3^2x_4^2 + 2x_3x_4^3 + 3x_4^4\) \\
			 
			 10 & \( 3x_1^4 + 2x_1^3x_2 + x_1^3x_3 + x_1^3x_4 + 4x_1^2x_2x_3 + 2x_1^2x_2x_4 +
             5x_1^2x_3x_4 + 6x_1^2x_4^2 + x_1x_2^3 + 2x_1x_2^2x_4 + 5x_1x_2x_3^2 + 3x_1x_2x_3x_4 +
             4x_1x_2x_4^2 + 5x_1x_3^3 + x_1x_3^2x_4 + x_1x_3x_4^2 + x_1x_4^3 + 6x_2^4 + x_2^3x_4 +
             6x_2^2x_3^2 + x_2^2x_3x_4 + 4x_2^2x_4^2 + x_2x_3^3 + 5x_2x_4^3 + 2x_3^4 + 5x_3^3x_4 +
             5x_3^2x_4^2 + x_3x_4^3 + 6x_4^4\) \\
			 
             \(\infty\) & \( 3x_1^4 + 3x_1^3x_2 + 3x_1^3x_3 + 6x_1^2x_2^2 + 3x_1^2x_2x_4 + 2x_1^2x_3^2 + 2x_1^2x_3x_4 + 3x_1^2x_4^2 + 6x_1x_2^3 + 5x_1x_2^2x_3 + x_1x_2x_3x_4 + 5x_1x_2x_4^2 + 5x_1x_3^3 + 4x_1x_3^2x_4 + 3x_1x_3x_4^2 + 6x_1x_4^3 + x_2^4 + 4x_2^3x_4 + 3x_2^2x_3^2 + 5x_2^2x_3x_4 + 5x_2x_3^3 + x_2x_3^2x_4 + 6x_2x_3x_4^2 + x_3^3x_4 + x_3^2x_4^2 + 3x_3x_4^3 + 4x_4^4\) \\
             \bottomrule
		\end{tabular}
	\end{center}
	\caption{Quartic K3 surfaces with specified Artin--Mazur height over \(\mathbb{F}_{7}\)}
\end{figure}

\begin{figure}[htbp]
	\begin{center}
		\def\arraystretch{1.5}
		\begin{tabular}{p{0.1\linewidth}p{0.8\linewidth}}
			 \toprule
             \textbf{Height} & \textbf{Equation} \\
			 \midrule
			 1 & $4x_1^4 + 6x_1^3x_2 + x_1^3x_3 + 2x_1^3x_4 + 3x_1^2x_2^2 + x_1^2x_2x_3 + 3x_1^2x_2x_4 + 6x_1^2x_3^2 + 6x_1^2x_3x_4 + 8x_1^2x_4^2 + 7x_1x_2^3 + 2x_1x_2^2x_3 + 8x_1x_2^2x_4 + 8x_1x_2x_3x_4 + 10x_1x_2x_4^2 + 10x_1x_3^3 + 9x_1x_3^2x_4 + 6x_1x_3x_4^2 + 3x_1x_4^3 + 6x_2^4 + 7x_2^3x_3 + 4x_2^3x_4 + 10x_2^2x_3^2 + 3x_2^2x_3x_4 + 5x_2^2x_4^2 + 4x_2x_3^2x_4 + 6x_2x_4^3 + 3x_3^4 + 4x_3^3x_4 + 7x_3^2x_4^2 + 9x_3x_4^3 + 5x_4^4 $\\
			  
			 2 & $4x_1^4 + 5x_1^3x_2 + 9x_1^3x_3 + 2x_1^3x_4 + 8x_1^2x_2^2 + x_1^2x_2x_3 + 9x_1^2x_2x_4 + x_1^2x_3^2 + 8x_1^2x_3x_4 + 6x_1x_2^3 + 10x_1x_2^2x_3 + 2x_1x_2^2x_4 + 10x_1x_2x_3^2 + 9x_1x_2x_3x_4 + 6x_1x_2x_4^2 + 8x_1x_3^3 + 4x_1x_3^2x_4 + 7x_1x_3x_4^2 + 9x_1x_4^3 + 3x_2^4 + 7x_2^3x_3 + 6x_2^3x_4 + 10x_2^2x_3^2 + 8x_2^2x_3x_4 + x_2^2x_4^2 + 9x_2x_3^3 + 6x_2x_3^2x_4 + x_2x_3x_4^2 + 9x_3^4 + 10x_3^3x_4 + x_3^2x_4^2 + x_3x_4^3 + 4x_4^4$\\
			  
			 3 & $10x_1^4 + 9x_1^3x_2 + 5x_1^3x_3 + 4x_1^3x_4 + 3x_1^2x_2^2 + 9x_1^2x_2x_3 + 4x_1^2x_2x_4 + 10x_1^2x_3^2 + 4x_1^2x_3x_4 + 8x_1^2x_4^2 + 8x_1x_2^3 + 9x_1x_2^2x_3 + 3x_1x_2^2x_4 + 7x_1x_2x_3^2 + 3x_1x_2x_4^2 + 8x_1x_3^3 + 2x_1x_3^2x_4 + x_1x_3x_4^2 + 7x_1x_4^3 + 2x_2^4 + 3x_2^3x_4 + x_2^2x_3^2 + x_2^2x_3x_4 + x_2^2x_4^2 + 5x_2x_3^3 + 9x_2x_3^2x_4 + 9x_2x_3x_4^2 + 4x_2x_4^3 + 5x_3^4 + 10x_3^3x_4 + 10x_3x_4^3 + 10x_4^4$\\
             
			 4 & $2x_1^4 + 4x_1^3x_2 + 9x_1^3x_3 + 10x_1^3x_4 + 2x_1^2x_2^2 + 4x_1^2x_2x_3 + 4x_1^2x_2x_4 + 4x_1^2x_3^2 + 10x_1^2x_3x_4 + 9x_1^2x_4^2 + 5x_1x_2^3 + 5x_1x_2^2x_3 + x_1x_2^2x_4 + 8x_1x_2x_3^2 + 2x_1x_2x_3x_4 + 10x_1x_2x_4^2 + 8x_1x_3^3 + 7x_1x_3^2x_4 + 5x_1x_3x_4^2 + 4x_1x_4^3 + 3x_2^4 + 6x_2^3x_3 + 4x_2^3x_4 + 10x_2^2x_3^2 + 5x_2^2x_3x_4 + 5x_2^2x_4^2 + x_2x_3^3 + 5x_2x_4^3 + 5x_3^4 + 7x_3^2x_4^2 + 5x_3x_4^3 + 9x_4^4$\\
			 
			 5 & \(10x_1^4 + x_1^3x_2 + 6x_1^3x_3 + 3x_1^3x_4 + x_1^2x_2^2 + 9x_1^2x_2x_3 + 6x_1^2x_2x_4 + 6x_1^2x_3^2 + 8x_1^2x_3x_4 + 4x_1^2x_4^2 + 3x_1x_2^3 + 7x_1x_2^2x_3 + 3x_1x_2^2x_4 + 7x_1x_2x_3^2 + 9x_1x_2x_3x_4 + 8x_1x_2x_4^2 + 7x_1x_3^3 + x_1x_3x_4^2 + 7x_1x_4^3 + x_2^4 + 3x_2^3x_3 + 7x_2^3x_4 + 5x_2^2x_3^2 + 7x_2^2x_3x_4 + 8x_2^2x_4^2 + 8x_2x_3^3 + 5x_2x_3^2x_4 + x_2x_3x_4^2 + 9x_2x_4^3 + 7x_3^4 + 4x_3^3x_4 + 4x_3^2x_4^2 + 3x_3x_4^3\)\\
             \bottomrule
		\end{tabular}
	\end{center}
	\caption{Quartic K3 surfaces with specified Artin--Mazur height over \(\mathbb{F}_{11}\)}
\end{figure}

\begin{figure}[htbp]
	\begin{center}
		\def\arraystretch{1.5}
		\begin{tabular}{p{0.1\linewidth}p{0.8\linewidth}}
			 \toprule
             \textbf{Height} & \textbf{Equation} \\
			 \midrule
			 1 & \(6x_1^4 + 7x_1^3x_3 + 4x_1^3x_4 + 6x_1^2x_2^2 + 7x_1^2x_2x_3 + 9x_1^2x_2x_4 + 2x_1^2x_3^2 + 3x_1^2x_3x_4 + 12x_1^2x_4^2 + 8x_1x_2^3 + 4x_1x_2^2x_3 + x_1x_2^2x_4 + 9x_1x_2x_3^2 + 8x_1x_2x_3x_4 + 10x_1x_2x_4^2 + 8x_1x_3^3 + 2x_1x_3^2x_4 + 9x_1x_3x_4^2 + 4x_1x_4^3 + 5x_2^4 + 4x_2^3x_3 + 2x_2^2x_3^2 + x_2^2x_3x_4 + 2x_2^2x_4^2 + 10x_2x_3^3 + 2x_2x_3^2x_4 + 2x_2x_3x_4^2 + 5x_2x_4^3 + 4x_3^4 + 3x_3^2x_4^2 + 2x_4^4\)\\
			  
			 2 & \(12x_1^4 + 8x_1^3x_2 + 8x_1^3x_3 + 10x_1^3x_4 + 8x_1^2x_2^2 + 11x_1^2x_2x_4 + 8x_1^2x_3^2 + 12x_1^2x_3x_4 + x_1^2x_4^2 + 7x_1x_2^3 + 9x_1x_2^2x_3 + 11x_1x_2^2x_4 + 10x_1x_2x_3^2 + 7x_1x_2x_4^2 + 8x_1x_3^3 + 3x_1x_3^2x_4 + 11x_1x_3x_4^2 + x_1x_4^3 + 4x_2^4 + 7x_2^3x_3 + 4x_2^3x_4 + 8x_2^2x_3^2 + 12x_2^2x_3x_4 + 6x_2^2x_4^2 + 7x_2x_3^3 + 12x_2x_3^2x_4 + 4x_2x_3x_4^2 + 10x_2x_4^3 + 4x_3^4 + 8x_3^3x_4 + 5x_3^2x_4^2 + 4x_3x_4^3\)\\
			  
			 3 & \(8x_1^4 + 2x_1^3x_2 + 3x_1^3x_3 + x_1^3x_4 + 6x_1^2x_2^2 + 7x_1^2x_2x_3 + 5x_1^2x_2x_4 + 2x_1^2x_3^2 + x_1^2x_4^2 + 11x_1x_2^3 + 10x_1x_2^2x_3 + 3x_1x_2^2x_4 + 5x_1x_2x_3^2 + 10x_1x_2x_3x_4 + 7x_1x_2x_4^2 + 12x_1x_3^3 + 12x_1x_3^2x_4 + 5x_1x_3x_4^2 + 7x_1x_4^3 + 7x_2^4 + 6x_2^3x_3 + 3x_2^3x_4 + 10x_2^2x_3^2 + 5x_2^2x_3x_4 + 12x_2^2x_4^2 + x_2x_3^3 + 3x_2x_3^2x_4 + 12x_2x_3x_4^2 + 8x_2x_4^3 + 10x_3^4 + 7x_3^3x_4 + 4x_3^2x_4^2 + 8x_3x_4^3 + 2x_4^4\)\\
             
			 4 & \(4x_1^4 + 4x_1^3x_2 + 2x_1^3x_3 + 3x_1^3x_4 + 9x_1^2x_2^2 + 6x_1^2x_2x_3 + 7x_1^2x_2x_4 + 10x_1^2x_3^2 + x_1^2x_3x_4 + 4x_1x_2^3 + 4x_1x_2^2x_3 + 6x_1x_2^2x_4 + 12x_1x_2x_3^2 + 7x_1x_2x_3x_4 + 3x_1x_2x_4^2 + 11x_1x_3^3 + 9x_1x_3^2x_4 + 10x_1x_3x_4^2 + 11x_1x_4^3 + 3x_2^4 + 5x_2^3x_3 + 8x_2^3x_4 + 5x_2^2x_3x_4 + 5x_2^2x_4^2 + 5x_2x_3^3 + 10x_2x_3^2x_4 + 2x_2x_3x_4^2 + 10x_2x_4^3 + 4x_3^4 + 5x_3^2x_4^2 + 4x_3x_4^3 + 6x_4^4\)\\
			  
			 5 & \(11x_1^4 + 4x_1^3x_2 + 12x_1^3x_3 + 4x_1^3x_4 + 6x_1^2x_2^2 + 10x_1^2x_2x_3 + 4x_1^2x_2x_4 + x_1^2x_3^2 + 7x_1^2x_3x_4 + 4x_1^2x_4^2 + 6x_1x_2^3 + 11x_1x_2^2x_3 + 7x_1x_2^2x_4 + 8x_1x_2x_3^2 + 10x_1x_2x_4^2 + x_1x_3^3 + 9x_1x_3^2x_4 + 8x_1x_3x_4^2 + 11x_1x_4^3 + 4x_2^4 + 8x_2^3x_3 + 5x_2^2x_3x_4 + 7x_2^2x_4^2 + 8x_2x_3^3 + 6x_2x_3^2x_4 + 5x_2x_4^3 + 2x_3^4 + 10x_3^3x_4 + 8x_3^2x_4^2 + 10x_3x_4^3 + 6x_4^4\) \\
            \bottomrule
		\end{tabular}
	\end{center}
	\caption{Quartic K3 surfaces with specified Artin--Mazur height over \(\mathbb{F}_{13}\)}
\end{figure}

\clearpage

\bibliographystyle{plain}
\bibliography{main}

\end{document}